\documentclass[11pt]{amsart}
\usepackage{amsmath}
\usepackage{bbm}
\usepackage{mathrsfs}
\usepackage[english]{babel}
\usepackage[utf8x]{inputenc}
\usepackage[T1]{fontenc}
\usepackage{wrapfig}

\usepackage{ dsfont }

\usepackage[usenames,dvipsnames]{pstricks}
\usepackage{epsfig}
\usepackage{pst-grad} 
\usepackage{pst-plot} 

\usepackage[top=3cm,bottom=2cm,left=3cm,right=3cm,marginparwidth=1.75cm]{geometry}

\usepackage{amsmath,amssymb}
\usepackage{graphicx}
\usepackage[colorinlistoftodos]{todonotes}
\usepackage[colorlinks=true, allcolors=blue]{hyperref}
\usepackage[capitalise]{cleveref}

\usepackage{cite}
\usepackage{enumitem} 
\usepackage[all,cmtip]{xy}
\usepackage{tikz-cd}
\usepackage{lmodern}
\usepackage{pgfplots}
\usetikzlibrary{intersections}
\usepackage{relsize}

\usepackage{youngtab}

\newtheorem{theorem}{Theorem}
\newtheorem{proposition}[theorem]{Proposition}

\theoremstyle{remark} 
\newtheorem{remark}[theorem]{Remark}
\newtheorem{example}[]{Example}

\newcommand{\R}{\mathbb{R}}
\newcommand{\G}{\mathbb{G}}

\newcommand{\C}{\mathbb{C}}
\newcommand{\N}{\mathbb{N}}
\newcommand{\Z}{\mathbb{Z}}

\newcommand{\RP}{\mathbb{R}\mathrm{P}}
\newcommand{\CP}{\mathbb{C}\mathrm{P}}

\newcommand{\EE}{\mathbb{E}}
\newcommand{\PP}{\mathbb{P}}

\DeclareMathOperator{\vol}{vol}

\newcommand{\be}{\begin{equation}}
\newcommand{\ee}{\end{equation}}

\numberwithin{equation}{section}

\title{What is... Random Algebraic geometry?}

\author{Antonio Lerario}

\begin{document}

\begin{abstract}We survey some ideas from the subject of Random Algebraic Geometry, a field that introduces a probabilistic perspective on classical topics in real algebraic geometry. This offers a modern approach to classical problems, such as Hilbert's Sixteenth Problem.
\end{abstract}
\maketitle
\section*{Introduction}

Inspired by various conversations with friends and collaborators, I decided to write these introductory notes on the emerging field of \emph{Random Algebraic Geometry}. They reflect my personal perspective on the subject and are not intended to be exhaustive. For example, my ``proofs'' are merely sketches—they can be made rigorous with some additional effort on the reader's part. Rather than focusing on strict rigor, I aim to highlight the main underlying ideas and how they are broadly connected. As a result, this is a non-technical paper, and I hope you will enjoy reading it. 

The paper is structured as follows.

In \cref{sec:genericvsrandom}, I examine the measure-theoretic aspects of the classical algebraic geometry notion of ``generic'' and motivate the introduction of randomness in real algebraic geometry. This includes a discussion of basic probabilistic concepts and natural methods for endowing parameter spaces—specifically Grassmannians and spaces of polynomials—with probability distributions.

In \cref{sec:degreevsvolume}, I compare the notions of degree and volume, demonstrating how they translate into each other when adopting a probabilistic viewpoint via integral geometry. As an application, I sketch proofs of Bernstein–Khovanskii–Kouchnirenko's Theorem on the number of solutions of polynomial systems with specified monomial supports, and of Edelman–Kostlan–Shub–Smale's Theorem on the expected number of real zeros of a random polynomial system. I also discuss applications of the integral geometry approach to the so-called probabilistic Schubert Calculus, introduced by Bürgisser and myself.

Finally, in \cref{sec:topology}, I discuss a random version of Hilbert's Sixteenth Problem concerning the topology of real algebraic hypersurfaces. Here, I sketch the proof of Gayet–Welschinger's Theorem on the expected Betti numbers of random real algebraic projective hypersurfaces and present a low-degree approximation theorem proved by Diatta and myself.

\section{Generic and random}\label{sec:genericvsrandom}
\subsection{Complex and real discriminants}\label{sec:disc}We start these notes with some topological considerations on discriminants.  Suppose we are given a smooth manifold $P$  and a closed subset $\Sigma\subset P$ which is a submanifold--complex, i.e. a set that admits a nice stratification\footnote{In all the cases we will be interested in, $P$ will be a semialgebraic set and $\Sigma$ a semialgebraic subset. Still, some of the ideas we discuss in this section apply to more general situations, in which cases the notion of submanifold--complex that I have in mind is from \cite[Chapter 3, Exercise 15]{Hirsch}.} into finitely many smooth submanifolds, called strata. The dimension of $\Sigma$ is the maximum of the dimensions of its strata. For us $P$ will play the role of a ``parameter space'' and $\Sigma$ will be a ``discriminant'', meaning that it consists of the set of parameters which fail to have some property. 

\begin{example}[Polynomials of degree two]\label{ex:two}Take as a parameter space the set of real homogeneous polynomials of degree two in two variables:
\be\label{eq:polyBW} P=\left\{a_{20}x_0^2+a_{11}\sqrt{2}x_0x_1+a_{02}x_1^2\,\big|\, a_{20}, a_{11}, a_{02}\in \R\right\}.\ee
(There is a geometric reason for multiplying the coefficient of $x_0x_1$ by $\sqrt{2}$, this will become clear soon, see  \cref{sec:proba} below.)
To every element of $P$ we can associate its zeroes on $\RP^1$: the polynomial can have two projective zeroes, one zero (with multiplicity two), no zeroes, or it can vanish on the whole $\RP^1$ (in the case of the zero polynomial). A natural property which we can consider for the elements of 
$P$ is therefore ``having regular zeroes'', and the discriminant in this case will be the set of polynomials which fail to have this property\footnote{For us ``having no zeroes'' is a subcase of ``having regular zeroes'', by the red herring principle.}, i.e. polynomials with multiple (or infinitely many) zeroes on $\RP^1$:
$$\Sigma=\left\{a_{20}x_0^2+a_{11}\sqrt{2}x_0x_1+a_{02}x_1^2\,\big|\,2a_{20}a_{02}-a_{11}^2=0\right\}.$$
This discrimiant is a submanifold--complex of $P\simeq \R^3$, and it can be stratified into the two smooth strata $\{2a_{20}a_{02}-a_{11}^2=0\}\backslash \{0\}$, of dimension two, and the single point $\{0\}$. In particular $\Sigma$ is a hypersurface -- in the sense of real geometry. This hypersurface separates the whole space $P$ into three connected open sets, which we call \emph{chambers}: 
\be\label{eq:chambers} P\setminus \Sigma=\left\{2a_{20}a_{02}<a_{11}^2\right\}\sqcup\left\{2a_{20}a_{02}>a_{11}^2, a_{20}>0\right\}\sqcup\left\{2a_{20}a_{02}>a_{11}^2, a_{20}<0\right\}.\ee
The discriminant here acts as a ``wall'': we cannot go from one chamber to another one without crossing it, see \cref{fig:discquad}.

\begin{figure}\includegraphics{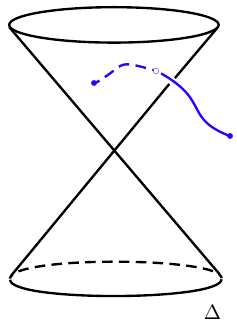}\caption{The discriminant in the space $P\simeq\R^3$ of real homogeneous polynomials of degree two is a quadric surface given by the equation $2a_{20}a_{02}-a_{11}^2=0$. Inside the cone we have $2a_{20}a_{02}>a_{11}^2$ and outside the cone $2a_{20}a_{02}<a_{11}^2$. Along a continuous path (the blue curve in the picture) between two polynomials, one inside the cone and the other outside the cone, we must hit $\Sigma$. }  \label{fig:discquad}\end{figure}
A similar construction can be made by taking as a parameter space the set of complex homogeneous polynomials of degree two (this space is just $\C^3$, the space of coefficients). We can still consider the property of ``having regular zeroes'', and the discriminant in this case consists of the set of complex polynomials with multiple or infinitely many zeroes on $\CP^1$:
$$ \Sigma_\C=\left\{a_{20}z_0^2+a_{11}\sqrt{2}z_0x_1+a_{02}z_1^2\,\big|\, 2a_{20}a_{02}-a_{11}^2=0\right\},$$
This discriminant is  a submanifold--complex of the space of complex polynomials, and it is made of the two smooth strata $\{2a_{20}a_{02}-a_{11}^2=0\}\backslash \{0\}$, of real dimension four, and the single point $\{0\}$. Clearly $\Sigma_\C\subset \C^3$ is a hypersurface in the complex sense (i.e. it is a complex algebraic subset of complex codimension one), but it is not a hypersurface in the real sense.
The main difference with the real case is that now the complement of $\Sigma_\C$ is connected, since its \emph{real} codimension is two.  In the space of complex polynomials we can continuously move between two polynomials with regular zeroes without hitting the discriminant.
\end{example}
Going back to the general picture, we see that what makes the difference between the real and the complex case is the codimension of the discriminant. It is a general fact that if it has codimension one it can\footnote{When $P=S^n$ and $\Sigma$ is a real algebraic set of codimension one, it \emph{must} separate. In fact, applying Alexander duality for the sphere, and working with $\Z_2$--coefficients, we get
$$ \widetilde{H}_0(S^n\setminus \Sigma)\simeq \widetilde{H}^{n-1}(\Sigma).$$
Since $\Sigma$ is a real algebraic set, by \cite[Proposition 11.3.1]{BCR:98}, it has a fundamental class over $\Z_2$ and $b_{n-1}(\Sigma)\neq 0$. In particular this implies that $b_0(S^n\setminus \Sigma)=1+\tilde{b}_0(S^n\setminus \Sigma)=1+b_{n-1}(\Sigma)>1.$
}
 separate $P$ into several connected open sets; if it has codimension at least two it cannot separate it. This is a purely topological fact, that is worth being recorded. 
\begin{proposition}\label{propo:codim}Let $P$ be a smooth manifold and $\Sigma\subset P$ be a closed submanifold--complex of codimension at least two. Then the inclusion $P\setminus \Sigma\hookrightarrow P$ induces an isomorphim on the zero--th homology.
\end{proposition}
\begin{proof}This is an easy consequence of Alexander duality. In fact, denoting by $n=\dim(P)$, by \cite[Theorem 3.44]{Hatcher}, working with $\Z_2$ coefficients\footnote{This ensures that $P$ is $\Z_2$--orientable.}, for every $k\in \mathbb{Z}$:
$$H_{k}(P, P\setminus \Sigma)\simeq H^{n-k}(\Sigma).$$
Since $\dim(\Sigma)<n-1$, then $H_{n}(\Sigma)= H_{n-1}(\Sigma)=0$ and consequently $H_1(P, P\setminus \Sigma)=H_0(P, P\setminus \Sigma)=0$. Plugging this into the long homology exact sequence of the pair $(P, P\setminus \Sigma)$, it gives $H_0(P\setminus \Sigma)\simeq H_0(P)$.
\end{proof}

The typical situation that one encounters in algebraic geometry is the study of the properties of a family of objects. The properties we will be interested in will be mostly topological (see  \cref{ex:hypersurfaces}).  In this context, I will now formulate a theorem which can be applied in many interesting situations and that captures the main features of the problems we will deal with. I will say that a map $f:X\to Y$ is algebraic if both $X$ and $Y$ are real (respectively complex) algebraic varieties and $f$ is real (respectively complex) algebraic.
\begin{theorem}\label{thm:disc}Let $f:X\to P$ be an algebraic map between smooth and compact varieties, real or complex. There exists a submanifold--complex $\Sigma\subset P$, of real codimension at least one in the real case and of real codimension at least two in the complex case, such that:
$$ f|_{X\setminus f^{-1}(\Sigma)}:X\setminus f^{-1}(\Sigma)\to P\setminus \Sigma$$
 is a topological fibration. \end{theorem}
 \begin{proof} First, both in the real and the complex case, the set $R$ of regular values of $f$ is dense in $P$, by Sard's Lemma\footnote{See \cite[Theorem 9.6.2]{BCR:98} for a semialgebraic version of the statement.}. In the real algebraic case this set is semialgebraic, and in the complex case it is constructible (in both cases it is the algebraic image of an algebraic subset). Let $\Sigma$ be the complement of $R$. Then $\Sigma$ is semialgebraic (constructible in the complex case) and closed, since it is the continuous image of a compact set (the set of points where the rank of $f$ is not maximal, a closed set in the compact set $X$). In the real case $P\setminus \Sigma$ is open dense and therefore $\dim(\Sigma)<\dim(P)$. In the complex case $\Sigma$ is constructible, in particular it can be stratified into constructible pieces, which are locally closed complex algebraic subsets; since $P\setminus \Sigma$ is open dense, no such piece can have full complex dimension, and in particular the complex codimension of $\Sigma$ is at least one, and its real codimension at least two. The fact that $f|_{X\setminus f^{-1}(\Sigma)}$ is a topological fibration follows now from Ehresmann's Theorem \cite{Ehresmann}. 
 \end{proof}
 Let $P$ be a smooth manifold. We say that a property is ``generic'' if it is valid for all the elements of $P$ except possibly for those elements belonging to a  subset of zero measure. This is a measure--theoretic translation of the corresponding notion from complex algebraic geometry. For instance, in the hypothesis of \cref{thm:disc}, one can say that \emph{the generic fiber of $f$ is smooth}. Moreover, in the complex case, combining \cref{thm:disc} and  \cref{propo:codim}, we see that something stronger is true: there exists a generic topological type for the fibers of $f$. This will in general be false in the real case.

 \begin{example}[Projective hypersurfaces of degree $d$]\label{ex:hypersurfaces} Let $P$ be the projectivization of the space of polynomials of degree $d$ in $n+1$ variables and $Y$ be the $n$--dimensional projective space. Let us not make a distinction between the real and the complex case, for now. We define $X$ to be the ``universal hypersurface'':
 $$X=\left\{(y,p)\in Y\times P\,\big|\, p(y)=0\right\}.$$
 and we introduce the map $f:X\to P$ as the projection onto the second factor. Notice that for a given $p\in P$, the fiber $f^{-1}(p)$ is homeomorphic to the zero set of $p$. We are in the position of applying  \cref{thm:disc}: there exists a closed submanifold--complex $\Sigma\subset P$ such that $f|_{X\setminus f^{-1}(\Sigma)}:X\setminus f^{-1}(\Sigma)\to P\setminus \Sigma$ is a topological fibration. This in particular means that the fibers of $f$ over each connected component of $P\setminus \Sigma$ are all homemorphic. 
 
 In the complex case, the real codimension of $\Sigma$ is two and, by  \cref{propo:codim}, $P\setminus \Sigma$ is connected. There is a generic topological type for the fibers of $f$, i.e. there is a generic topological type of hypersurfaces of degree $d$ in $\CP^n.$
 
In the real case the discriminant partitions the space of real polynomials into many chambers, which are called \emph{rigid isotopy classes}. In this context, the first part of Hilbert's Sixteenth Problem \cite{H16trans, Wilson} asks for the study of the maximal number and the possible arrangement of the components of a nonsingular real algebraic
hypersurface of degree $d$ in $\RP^n$. We see that this is essentially the problem of understanding the geometry of the rigid isotopy classes. In some sense it is a completely out of reach problem, if approached in a deterministic way. Already in the case $d=2$, Kharlamov and Orevkov \cite{OreKha} have proved that there exists $c>0$ such that the number of rigid isotopy classes (i.e. the number of connected components of $P\setminus \Sigma$) of real plane curves of degree $d$ is bigger than $e^{cd^2}$ (already in the case $d=8$ there are more then 2500 such components!)
 \end{example}
 
 The observation that complex discriminants have real codimension two, and therefore cannot separate, whereas real discriminants have real codimension one and therefore \emph{can} separate is at the origin of the notion of ``generic'' from complex algebraic geometry; at the same time, it motivates the interest into probabilistic questions in real algebraic geometry.

\subsection{From generic to random}The transition from classical algebraic geometry to random algebraic geometry happens by insisting on the measure--theoretic aspect of the notion of generic. Suppose for instance that our parameter space $P$ also comes with a reasonable probability measure $\mu$, where by ``reasonable'' we mean that $\mu$ is absolutely continuous with respect to the Lebesgue measure\footnote{The property of being absolutely continuous with respect to the Lebesgue measure is well defined for a measure on a smooth manifold $P$, for the same reason for which the notion of set of measure zero is well defined (even if no preferred Lebesgue measure exists on $P$).}. Then, given a discriminant $\Sigma\subset P$ separating $P$ into many chambers it makes sense to study the measure of these chambers. Phrasing this into probabilistic language: what is the probability that a point $p\in P$ sampled at random from this measure belongs to a given chamber? Similarly, given a smooth map $f:X\to P$, with $X$ compact\footnote{Compactness of $X$ ensures that the regular fibers of $f$ have finite dimensional cohomology, but it is not essential -- having compact fibers would be enough, for instance.}, one might ask for the expectation of the topology (measured by the Betti numbers\footnote{Given a topological space $X$ we denote its $k$--th Betti number by $b_k(X)=\dim(H_{k}(X,\Z_2))$ and its total Betti number by $b(X)=\sum_{k\geq 0}b_k(X)$. For us $X$ will be semialgebraic, and in particular its Betti numbers will be finite.}, for instance) of the fibers of $f$, which in this language means computing the following integral:
\be \label{eq:mean}\underset{p\in P}{\EE}b(f^{-1}(p)):=\int_{P}b(f^{-1}(p))\mathrm{d}\mu \ee
Notice that, in the hypothesis of  \cref{thm:disc}, the function $p\mapsto b(f^{-1}(p))$ is integrable, since it is bounded. This is a consequence of the fact that $P\setminus \Sigma$ has finitely many components and that each regular fiber of $f$ has finite dimensional cohomology.
Clearly, if $P\setminus \Sigma$ consists of one single chamber, as it happens in the complex case, this chamber has probability one and the above expectation equals the sum of the Betti numbers of the generic fiber, no matter the choice of the reasonable measure.

The random approach offers a new point of view on some classical questions in real algebraic geometry. For instance, in the context of \cref{ex:hypersurfaces}, one might formulate a probabilistic version of Hilbert's Sixteenth Problem: what is the expected number of connected components of a random hypersurface of degree $d$ in $\RP^n$? And what is the structure of the arrangement of these components? To give a precise formulation of the first of these two questions we can proceed as follows: we endow the space $P$ of polynomials of degree $d$ in $n+1$ variables with a reasonable probability distribution $\mu$, and then ask for the expectation of the random variable $b_0(Z(p))$, i.e. we ask for the value of the integral $\int_P b_0(Z(p))\mathrm{d}\mu$. The second question requires a little bit more of work, and we postpone its discussion to \cref{sec:topology} below. 

Of course, the answer to questions like these depends on the choice of the probability measure on the space $P$. This is a delicate point of the subject but, as we will see, in most cases there are some reasonable choices that also have some geometric meaning.

\subsection{A little of Probability}\label{sec:proba}I will recall here some basic facts from probability theory. As the reader will see, these are completely elementary and no hard material is needed for the moment -- I will introduce additional material as we go further in the subject.

Recall that a probability space $(P, \mathcal{F}, \mu)$ is just a measure space (i.e. a sample space $P$, a sigma--algebra $\mathcal{F}$ of subsets of $P$ and a measure $\mu $ on $\mathcal{F}$) such that $\mu(P)=1$. If $h:P\to Q$ is a measurable map, the measure $h_\#\mu$ on $Q$ defined by $h_\#\mu(A):=\mu(h^{-1}(A))$ is called the \emph{pushforward measure}.  If the probability measure we are referring to is clear from the context, the probability of a measurable set $E\subseteq P$ will be denoted by $\PP(E)\in [0,1].$ If $X$ is a topological space, a random variable $\xi:P \to X$ is just a measurable map (i.e. the preimage under $\xi$ of any open set is a measurable set in $P$). We will often completely forget about this abstract setting and simply assume that an underlying probability space exists. However, sometimes it is useful, at least for acquiring confidence, to get back to these definitions. 

If a random variable $\xi:P\to \R$ is Lebesgue integrable with respect to the measure $\mu$, its  expectation is defined to be the Lebesgue integral 
$$\EE\xi:=\int_{P}\xi \mathrm{d}\mu.$$
In most of the cases we will be interested in bounded random variables, hence automatically integrable. Two random  variables $\xi_1, \xi_2:P\to X$ are said to have the same distribution if for every open set $A\subset X$ we have $\PP(\xi_1\in A)=\PP(\xi_2\in A)$; in this case we will write $\xi_1\sim \xi_2.$
The random variables $\xi_1, \ldots, \xi_n:P\to X$ are said to be independent if for all possible choices of open sets $A_1, \ldots, A_n\subset X$ we have
$$ \PP(\xi_1\in A_1, \ldots, \xi_n\in A_n)=\PP(\xi_1\in A_1)\cdots \PP(\xi_n\in A_n).$$
\subsubsection{Gaussian random variables}Gaussian variables are simple to study, natural and ubiquitous. In the context of random algebraic geometry, very little can be said outside the world of gaussian variables.  A \emph{standard gaussian variable} is a random variable $\xi:P\to \R$ such that for every open set $A\subseteq \R$
$$ \PP(\xi\in A)=\frac{1}{\sqrt{2\pi}}\int_{A}e^{-\frac{x^2}{2}}\mathrm{d}x.$$
In several variables, let $\xi=(\xi_1, \ldots, \xi_n)$ be a vector filled with independent standard gaussians. We denote by $\gamma_n$ the probability measure on $\R^n$ defined, for every measurable $A\subseteq \R^n$, by:
\be\label{eq:gaussianmulti} \gamma_n(A):=\frac{1}{(2\pi)^{\frac{n}{2}}}\int_{A}e^{-\frac{\|x\|^2}{2}}\mathrm{d}x.\ee
In this way $\R^n$ becomes a probability space, called the \emph{standard multivariate gaussian space}.

\begin{remark}[The gaussian distribution in polar coordinates]Let $j:S^{n-1}\hookrightarrow \R^{n}$ be the inclusion and consider the ``polar coordinates'' map $\psi:S^{n-1}\times (0, \infty)\to \R^n$ given by 
$$ \psi(\theta, \rho)=\rho j(\theta).$$ Denoting by ``$\mathrm{d}\theta$'' the integration on $S^{n-1}$ with respect to the standard volume form, and using the change of variables formula, we see that if $A\subseteq \R^n$ is an open set,
\be\label{eq:gaussian} \gamma_n(A)=\frac{1}{(2\pi)^{\frac{n}{2}}}\int_{A}e^{-\frac{\|x\|^2}{2}}\mathrm{d}x=\frac{1}{(2\pi)^{\frac{n}{2}}}\int_{\psi^{-1}(A)}\rho^{n-1}e^{-\frac{\|\rho\|^2}{2}}\mathrm{d}\rho \mathrm{d}\theta.\ee
It follows from \eqref{eq:gaussian} that, if $g\in O(n)$ is an orthogonal matrix,  the random variable $g\xi$ has the same distribution of $\xi$. This is relevant especially when we are interested in computing expectations as in \eqref{eq:mean}.
\end{remark}

We will use the standard gaussian space to put probability measures on other vector spaces. We can do this as follows: given  a finite--dimensional vector space $V$ and a surjective linear map $h:\R^n\to V$, we simply consider the measure $\mu=h_\#\gamma_n$. In this way $(V, \mu)$ becomes what is called a nondegenerate\footnote{If $h:\R^n\to V$ is not surjective,  still $h_\#\gamma_n$ is called a gaussian measure, but its support is a smaller subspace $V'\subset V$. We will only be interested in the nondegenerate case} \emph{gaussian space}
 (of course the probability measure just defined depends on $h$, no need to say).

\begin{remark}[Gaussian measures and scalar products]\label{remark:gaussianscalar}One interesting fact to observe is that the choice of a nondegenerate gaussian distribution on $V$ is equivalent to the choice of a scalar product on it. If $h:\R^n\to V$ is a linear map inducing $h_\#\gamma_n$, endowing $\R^n$ with the standard Euclidean structure, we write $\R^n=\ker(h)\oplus \ker(h)^\perp$ and use the fact that $V=\mathrm{im}(h)\simeq \ker(h)^\perp$ to put a scalar product on $V$. Viceversa, if $V$ is endowed with a scalar product $\langle\cdot,\cdot \rangle$, we can pick an orthonormal basis $\{v_1, \ldots, v_m\}$ for it and define $h:\R^m\to V$ by $h(e_j):=v_j$. Notice that, in this way, a random element from $V$ (i.e. a random variable with values in $V$ distributed according to the probability distribution $h_\#\gamma_m$) can be written as:
\be v=\xi_1 v_1+\cdots+\xi_m v_m,\ee
i.e. it is a linear combination of the basis element with random (gaussian) coefficients.
\end{remark}

\subsection{Random linear spaces}\label{sec:randomlinear}In this section we will put a probability distribution on the simplest families of objects we will encounter: Grassmannians, i.e. families of linear spaces. Let us start with some general observations. 

If our parameter space $P$ is endowed with a Riemannian metric $g$, recall that we can define on it a volume density, which we denote by $\mathrm{vol}_g$. For every Borel set $E\subseteq P$ we have a notion of volume, defined by $\mathrm{vol}_g(E):=\int_E \mathrm{d}\mathrm{vol}_g.$ If the space has finite total volume, i.e. if $\mathrm{vol}_{g}(P)<\infty$, then we can normalize this volume and turn $P$ into a probability space: for every $E\subseteq P$ measurable we define
$$\PP(E):=\frac{\mathrm{vol}_g(E)}{\mathrm{vol}_g(P)}.$$
We call the result of this construction the \emph{uniform distribution} on $P$ (of course this depends on the metric $g$ but, when clear from the context, we will omit the dependence on $g$ in the notation). 

For example, the uniform measure on $S^n$ is obtained by normalizing the volume coming from the Riemannian metric induced from $\R^{n+1}$. Similarly, since the antipodal map $x\mapsto -x$ is an isometry of $S^n$, the Riemannian metric induced from $\R^{n+1}$ descends to a Riemannian metric on $\RP^n$, which we will call the \emph{quotient metric}; normalizing its volume we get the uniform measure on $\RP^n$. Similarly, the uniform measure on the orthogonal group $O(n)$ is obtained by normalizing the volume coming from the Riemannian metric induced by $O(n)\hookrightarrow \R^{n\times n}$ (this is also called the \emph{Haar measure}).

Let us apply this idea in the case $P=\mathrm{G}(k,n)$ is the Grassmannian of $k$--dimensional vector subspaces of $\R^n$. First notice that we can view the Grassmannian as the quotient:
\be \mathrm{G}(k,n)=\frac{O(n)}{O(k)\times O(n-k)}.\ee
The subgroup $O(k)\times O(n-k)$ acts by left translations on $O(n)$ by isometries, and the Riemannian metric of $O(n)$ (induced by the standard Euclidean structure on the ambient space $\R^{n\times n}$) descends to a Riemannian metric on the quotient $\mathrm{G}(k,n).$ Normalizing the total volume induced by this metric,  we get the uniform distribution on $\mathrm{G}(k,n)$ and we can therefore introduce the notion of ``random $k$-dimensional linear space'': this is just a point sampled uniformly from $\mathrm{G}(k,n)$.
Clearly this can be done also for the Grassmannian $\G(k,n)$ of projective subspaces\footnote{Occasionally we will use the name \emph{$k$--flat} for a projective subspace of dimension $k$.} of dimension $k$ in $\RP^n$, since $\G(k,n)\simeq G(k+1, n+1).$

Notice that the uniform distribution on $\mathrm{G}(k,n)$ is invariant under the natural action of $O(n)$, and in fact it is the \emph{unique} probability distribution\footnote{Notice that, except for the case $(k,n)=(2,4)$ there is a unique (up to multiples) Riemannian metric on $G(k,n)$ which is invariant under the action of $O(n)$. The uniqueness of a probability distribution with such property holds instead for all $(k,n)$.} on $\mathrm{G}(k,n)$ which is invariant under this action. For practical purposes, a way to sample a random subspace uniformly from $\mathrm{G}(k,n)$ is the following: we pick a matrix $M\in \R^{n\times k}$ filled with independent standard Gaussian variables and we consider the span of its columns. With probability one this span is $k$--dimensional, hence the random matrix $M$ induces a probability distribution on $\mathrm{G}(k,n)$. Since for every $g\in O(n)$ the matrices $M$ and $gM$ have the same distribution, the induced distribution on $\mathrm{G}(k,n)$ is $O(n)$--invariant and therefore it coincides with the uniform distribution. 

Another equivalent way of defining the uniform distribution on $\mathrm{G}(k,n)$ is the following. Fix a linear space $\R^k\subseteq \R^n$. Normalizing the total volume of $O(n)$, we can think of it as a probability space and the map $g\mapsto g\R^k$ is therefore a random variable with values in $\mathrm{G}(k,n)$. The corresponding probability distribution is $O(n)$--invariant and consequently, by uniqueness, it is the uniform distribution.

\begin{remark}This construction can also be applied for the definition of a random \emph{complex} subspace, simply by replacing the orthogonal group with the unitary group, as we will do in \cref{sec:IGFC}. \end{remark}

\subsection{A measure on the space of polynomials}\label{sec:BWdef}We will now move to the case of polynomials, starting with an important remark. As we have seen, in the linear case there is essentially one natural probability distribution to consider, but this will no longer be true for polynomials of higher degree. Let me comment on the word ``natural'. We will be interested in the zero set of a homogeneous polynomial in projective space. Since there is no preferred points or directions in $\RP^n$, it is ``natural'' to require that our probability distribution is invariant under the action of the orthogonal group $O(n+1)$ by change of variables. 
There is an infinite dimensional family of such distributions! If we insist that the distribution should also be gaussian and nondegenerate, by \cref{remark:gaussianscalar}, looking for such distribution is equivalent to find a scalar product on the space of polynomials which is invariant under the action of $O(n+1)$ by change of variables. A probability measure, on the space of polynomials, satisfying these properties (gaussian, nondegenerate and invariant under change of variables) will be called an \emph{invariant measure}. Even in this case, there are many possibilities -- but now a finite--dimensional family depending on $\lfloor\frac{d}{2}\rfloor+1$ parameters (see \cref{sec:projection}). 

The measure that we will introduce now, which will be called the \emph{Bombieri--Weyl} measure (or the Kostlan measure), is special even among the family of invariant measures: it comes from the unique invariant measure on the space of complex polynomials. In order to introduce it we first consider the Bombieri--Weyl hermitian structure on the space of complex polynomials, defined for $p,q\in \C[z_0, \ldots, z_n]_{(d)}$ by
$$ \langle p, q\rangle_{\mathrm{BW}}:=\frac{1}{\pi^{n+1}}\int_{\C^{n+1}}p(z)\overline{q(z)}e^{-\|z\|^2}\mathrm{d} z,$$
where $\mathrm{d}z:=(i/2)^{n+1}\mathrm{d}z_0\mathrm{d}\overline{z_0}\cdots \mathrm{d}z_n\mathrm{d}\overline{z_n}$ is the Lebesgue measure.
It is not difficult to show that the Bombieri--Weyl hermitian structure is invariant under the action of $U(n+1)$ by change of variables. Moreover, since this action is irreducible (over $\C$), Schur's Lemma implies that this is the unique\footnote{In the real case, the non--uniqueness is a consequence of the fact that the orthogonal change of variables representation is not irreducible, see \cref{sec:projection}.} invariant hermitian structure (up to multiples) on the space of complex polynomials.

A hermitian orthonormal basis for the Bombieri--Weyl structure is given by 
\be\label{eq:BWbasis}\mathcal{B}_{n,d}:= \left\{\sigma_\alpha=\left(\frac{d!}{\alpha_0!\cdots \alpha_n!}\right)^{\frac{1}{2}}z_0^{\alpha_0}\cdots z_n^{\alpha_n}\right\}_{|\alpha|=d}.\ee
(I suggest that the reader tries to prove that different monomials are orthogonal.) 

We define the \emph{Bombieri--Weyl measure} as the probability measure $\gamma_{\mathrm{BW}}$ induced on the space of polynomials $\R[x_0, \ldots, x_n]_{(d)}$ by the random variable 
\be\label{eq:Kostlan1}p(x)=\sum_{|\alpha|=d}\xi_\alpha\cdot \left(\frac{d!}{\alpha_0!\cdots \alpha_n!}\right)^{\frac{1}{2}}x_0^{\alpha_0}\cdots x_n^{\alpha_n},\ee
where $\{\xi_\alpha\}$ is a family of independent standard gaussians. The wording ``Bombieri--Weyl polynomial'' will mean a random variable with values in $\R[x_0, \ldots, x_n]_{(d)}$ distributed as \eqref{eq:Kostlan1}.

Since the basis \eqref{eq:BWbasis} is real, and since $O(n+1)\subset U(n+1)$, the real part of the Bombieri--Weyl hermitian structure restricts to an invariant scalar product on $\R[x_0, \ldots, x_n]_{(d)},$ which we still call the \emph{Bombieri--Weyl} scalar product. Arguing as in \cref{remark:gaussianscalar}, the gaussian measure defined by \eqref{eq:Kostlan1} is the one associated to the Bombieri--Weyl scalar product and, since the latter is $O(n+1)$--invariant, so is the corresponding measure.

\begin{example}[The GOE ensemble]\label{example:GOE}In the special case $d=2$, the space of quadratic forms is isomorphic to the space of symmetric matrices, through the linear isomorphism 
\be\label{eq:isoQ}\varphi:\R[x_1, \ldots, x_n]_{(2)}\to \mathrm{Sym}(n, \R)\ee  that associates to every quadratic form $q$ the symmetric matrix $Q$ defined by $q(x)= \langle x,Qx\rangle$. Therefore the Bombieri--Weyl measure $\gamma_{\mathrm{BW}}$ induces a gaussian measure $\gamma_{\mathrm{GOE}}:=\varphi_\#\gamma_{\mathrm{BW}}$ on $\mathrm{Sym}(n, \R)$. The probability space $(\mathrm{Sym}(n, \R), \gamma_{\mathrm{GOE}})$ is called the \emph{Gaussian Orthogonal Ensemble}: it is an ensemble (i.e. a probability space) which is gaussian and invariant under the action of the orthogonal group by change of variables (which on the space of symmetric matrices is just the action by congruence). In some senses, Random Matrix Theory is degree--two Random Algebraic Geometry.
\end{example}

\section{Degree and volume}\label{sec:degreevsvolume}
\subsection{Integral geometry}\label{sec:integralgeometry}One of the most beautiful formulas in differential geometry, to my opinion, is the so called \emph{Integral Geometry Formula}, or kinematic formula. I learnt the subject from the monograph by R. Howard \cite{Howard}, whose reading is highly recommended. This formula was used for probabilistic purposes for the first time by A. Edelman and E. Kostlan in \cite{EdelmanKostlan95}. 

I will begin by discussing the projective version of such formula. First let me recall, following the discussion from \cref{sec:randomlinear}, that if $(P, g)$ is a Riemannian manifold and $j:A\hookrightarrow P$ is a submanifold, then $A$ inherits a Riemannian metric $j^*g$ and, consequently, we can define on it a volume density $\mathrm{vol}_{j^*g}$. If $\dim(A)<\dim(P)$, the total volume of $A$ with respect to $\mathrm{vol}_{g}$ is zero ($A$ is a zero--measure subset of $P$), but $\mathrm{vol}_{j^*g}(A)>0$. Below, for a submanifold $j:A\hookrightarrow P$ we will simply write ``$\mathrm{vol}(A)$'' for $\mathrm{vol}_{j^*g}(A)$. Also, we denote by ``$\mathrm{d}g$'' the integration on the orthogonal group with respect to the uniform measure.

\begin{theorem}[The Integral Geometry Formula for $\RP^n$]\label{thm:IGF}Let $A$ and $B$ be compact submanifolds of $\RP^n$ of dimensions $\dim(A)=a$ and $\dim(B)=b$ and denote by $c=n-(n-a)-(n-b)$. For almost all $g\in O(n+1)$ the intersection $A\cap gB$ is transversal and of dimension $c$, the function $g\mapsto\mathrm{vol}(A\cap gB)$ is measurable and
\be\label{eq:IGF} \int_{O(n+1)}\frac{\mathrm{vol}(A\cap gB)}{\mathrm{vol}(\RP^c)}\mathrm{d}g=\frac{\mathrm{vol}(A)}{\mathrm{vol}(\RP^a)}\frac{\mathrm{vol}(B)}{\mathrm{vol}(\RP^b)}.\ee
\end{theorem}
The fact that in the previous statement $A$ and $B$ are assumed to be compact is not really necessary, it is enough to assume that they are immersed submanifolds with finite total volume.  

We will see various declinations of the philosophy from \cref{thm:IGF}, ultimately dealing with integration over a compact group acting by isometries on a Riemannian homogeneous space, but the reader should be made aware that such a pleasant formula as in \eqref{eq:IGF} cannot exists in greater generality, unless some further assumptions are made on the submanifolds $A$ and $B$. 

Using the language of \cref{sec:randomlinear}, \eqref{eq:IGF} has a clear probabilistic interpretation. For instance, applying \eqref{eq:IGF} in the case $B=\RP^b$ and $a+b=n$, it tells that the expectation of the number of points of intersection between a submanifold $A\hookrightarrow \RP^n$ and a random flat of complementary dimension equals $\frac{\mathrm{vol}(A)}{\mathrm{vol}(\RP^a)}$.

\subsection{Newton polytopes}\label{sec:IGFC}
I would like now to make a short detour and introduce the complex version of the Integral Geometry Formula, for the purpose of giving a proof of the Bernstein--Khovanskii--Kushnirenko Theorem.

\begin{theorem}[Bernstein--Khovanskii--Kushnirenko\footnote{This statement is a special case of the more general situation in which the supports of the equations are different, i.e. when for every $j=1, \ldots, n$ in the $j$--th equation the monomials have exponents in $E_j\subset \Z^n$. In this case we get $n$ polytopes $P_{E_1}, \ldots, P_{E_n}\subset \R^n$, one for each equation, and the generic number of solutions in $(\C^*)^n$ of the corresponding system equals the mixed volume $\mathrm{MV}(P_{E_1}, \ldots, P_{E_n}).$}]\label{thm:BKK} Let $E\subset \Z^n$ be a finite set and denote by $P_E=\mathrm{conv}(E)\subset \R^n$ its convex hull. Consider the system of $n$ equations in $n$ variables:
\be\label{eq:BKK}\left\{\sum_{\eta\in E}c_{1, \eta}z_1^{\eta_1}\cdots z_n^{\eta_n}=\cdots=\sum_{\eta\in E}c_{n, \eta}z_1^{\eta_1}\cdots z_n^{\eta_n}=0\right\}.
\ee
For the generic choice of the coefficients $(c_{j, \eta})\in \C^{|E|\times n}$ the number of solutions of \eqref{eq:BKK} in $(\C^*)^n$ equals $n!\mathrm{vol}(P_E).$
\end{theorem}

Before proving \cref{thm:BKK}, let us establish a few interesting facts. First let us talk about the complex version of \eqref{eq:IGF}. We view the complex projective space $\CP^n$ as the quotient of $S^{2n+1}\subset \C^{n+1}$ by the action of $U(1)\simeq S^1$. Since this action is by isometries on $S^{2n+1}$ (with respect ot the Riemannian structure induced by the standard Hermitian structure on $\C^{n+1}$), the Riemannian metric of $S^{2n+1}$ descends to a quotient metric on $\CP^n$. The induced volume density on $\CP^n$ makes it of total volume:
\be \mathrm{vol}(\CP^n)=\frac{\pi^n}{n!}=\frac{\mathrm{vol}(S^{2n+1})}{\mathrm{vol}(S^1)}.\ee
With this construction, the unitary group $U(n+1)$ acts on $\CP^n$ by isometries; moreover if $Y\hookrightarrow \CP^n$ is a complex submanifold, then for every $g\in U(n+1)$ also $g Y\hookrightarrow \CP^n$ is a complex submanifold. If $Y\hookrightarrow \CP^n$ is a singular algebraic set, we define its volume by the volume of the set of its smooth points. This means: we restrict the Riemannian metric of $\CP^n$ to the set of smooth points of $Y$, turning it into a Riemannian manifold, and we consider the corresponding density.

Now comes the complex version of \eqref{eq:IGF}. Let $A, B\subset \CP^n$ be open subsets of algebraic sets of complex dimensions $\dim_{\C}(A)=a$ and $\dim_{\C}(B)=b$. Then, for almost every $g\in U(n+1)$ the intersection $A\cap gB$ is transversal of complex dimension $c=n-(n-a)-(n-b)$ and:
\be\label{eq:IGFC} \int_{U(n+1)}\frac{\mathrm{vol}(A\cap gB)}{\mathrm{vol}(\CP^c)}\mathrm{d}g=\frac{\mathrm{vol}(A)}{\mathrm{vol}(\CP^a)}\frac{\mathrm{vol}(B)}{\mathrm{vol}(\CP^b)}.\ee
Here ``$\mathrm{d}g$'' denotes the integration with respect to the normalized density induced by the inclusion $U(n+1)\hookrightarrow \C^{n\times n}$ as a Riemannian submanifold and $\mathrm{vol}(Y)$ denotes the $2\dim_{\C}(Y)$--dimensional volume of the complex submanifold $Y\hookrightarrow \CP^n$. 
\begin{figure}

\end{figure}
In order to get an exact formula like \eqref{eq:IGFC}, the requirement that $A, B$ are complex submanifolds of $\CP^n$ cannot be dropped: this is related to the fact that the stabilizer $U(n)$ of a point $p\in \CP^n$ under the action of $U(n+1)$ acts transitively on the set of complex directions in $T_p\CP^n$ (see \cref{sec:Schubert} for more details).

An interesting consequence of \eqref{eq:IGFC} is that it establishes a bridge between the notion of ``degree'' from complex algebraic geometry and the notion of ``volume'' from Riemannian geometry.
\begin{theorem}Let $A\subset \CP^n$ be an algebraic set of dimension $\dim_{\C}(A)=a$. Then:
\be \label{eq:volC}\deg(A)=\frac{\mathrm{vol}(A)}{\mathrm{vol}(\CP^a)}.\ee
\end{theorem}

\begin{proof}The proof is elementary: one simply applies \eqref{eq:IGFC} with the choice $B=\CP^{n-a}$. In this case the integrand on the left hand side of \eqref{eq:IGFC} is almost everywhere equal to $\deg(A)$ (by definition of degree as the number of points of intersection with a generic subspace of complementary dimension) and the right hand side of \eqref{eq:IGFC} reduces to \eqref{eq:volC}.
\end{proof}

Notice that, viewing the unitary group as a probability space, the map $g\mapsto g\CP^k$, for $g\in U(n+1)$, gives a random variable with values in the complex Grassmannian of $k$--flats in $\CP^n$, inducing on it the uniform distribution. In this way \eqref{eq:IGFC} also has a probabilistic interpretation, with the extra feature that in the complex framework ``random'' and ``generic'' are synonymous.

\begin{proof}[Proof of \cref{thm:BKK}]Consider the map $\nu_E:(\C^*)^n\to \CP^{|E|-1}$ given by:
\be \nu_E(z_1, \ldots, z_n)\mapsto [z_1^{\eta_1}\cdots z_n^{\eta_n}]_{\eta \in E}\ee
(this map is an analogue of the Veronese map from \cref{sec:veronese}). We denote by $T_E$ the projective closure of the image of $\nu_E$:
\be T_E=\overline{\mathrm{im} (\nu_E)}.\ee
The algebraic set $T_E$ is of dimension $n$ and we denote by $\CP^{|E|-1-n}$ the projective subspace, of complementary dimension, given by the vanishing of the last $n$ homogeneous coordinates. 
Observe now that for the generic choice of the coefficients $(c_{j, \eta})\in \C^{|E|\times n}$, i.e. whenever we can write the matrix of coefficients as $(c_{j, \eta})=g\cdot(\mathbf{1}_{|E|}|0)$ with $g\CP^{|E|-n-1}$ transversal to $T_E$ and with the points of intersection in $\mathrm{im}(\nu_E)$, we have: 
\be\label{eq:BKKint} \#\{\textrm{solutions of \eqref{eq:BKK}}\}=\#T_E\cap g\CP^{|E|-1-n}=\frac{n!}{\pi^n}\mathrm{vol}(T_E).\ee
 The second equality in \eqref{eq:BKKint} follows by an application of \eqref{eq:IGFC}, since the integrand is constant almost everywhere. 
 \begin{figure}\begin{center}
  \includegraphics{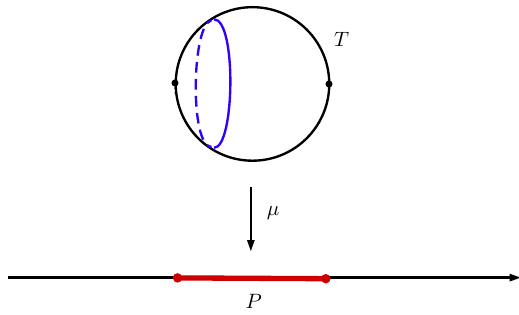}
  \end{center}
  \caption{In the case $E=\{0,1\}\subset \Z$ the embedding $\nu_E:\C^*\to \CP^1$ is given by $z\mapsto[1,z]$. In this case $T=\CP^1$ and the moment map $ \mu:\CP^{1}\to \R$ is given by $[w_0, w_1]\mapsto\frac{|w_1|^2}{|w_0|^2+|w_1|^2}$. The image of $\mu$ is the segment $[0,1]=\mathrm{conv}(E).$}
  \label{fig:moment}
\end{figure}
 In order to prove  \cref{thm:BKK} it remains to establish a connection between the volume of $T_E$ and the volume of the Newton polytope $P_E$. This comes from Hamiltonian geometry. Consider the map $F_E:\CP^{|E|-1}\to \R^n$ defined by:
\be F_E\left([w_\eta]_{\eta\in E}\right)=\frac{\sum_{\eta\in E} \eta |w_\eta|^2}{\sum_{\eta\in E} |w_\eta|^2}.\ee
This map takes value in $\R^n$ since $\eta\in \Z^n\subset \R^n$. We denote by $\mu_E$ the restriction of $F_E$ to $T_E$:
\be \mu_E:T_E\to \R^n.\ee
The map $\mu_E$ is called \emph{the moment map} (see  \cref{fig:moment}). Now comes the relevant fact:  the image of $\mu_E$ is $P_E$ \cite[Theorem 2]{Atyiah1} and  $\mu_E$ is ``volume preserving'' in the sense that for all $W\subseteq P_E$ measurable:
\be \label{eq:moment1}\mathrm{vol}(\mu_E^{-1}(W))=\pi^n\vol(W).\ee
Applying \eqref{eq:moment1} in the case $W=P_E$ we get $\mathrm{vol}(T_E)=\pi^n \mathrm{vol}(P_E)$ and, substituting this into \eqref{eq:BKKint}, it finally proves \cref{thm:BKK}. In other words, for the generic choice of the coefficients in the equations defining \eqref{eq:BKK}
\be \#\{\textrm{solutions of \eqref{eq:BKK}}\}=\frac{n!}{\pi^n}\mathrm{vol}(T_E)=n!\mathrm{vol}(P_E).\ee
\end{proof}
In the special case $E=\{(\eta_1, \ldots, \eta_n)\in \Z^n\,|\, 0\leq \eta_i\leq d\}$, the system \eqref{eq:BKK} consists of $n$ equations of degree $d$ with generic coefficients. In this case the Newton polytope $P_E$ has volume $\frac{d^n}{n!}$, the algebraic set $T_E\subset \CP^{\binom{n+d}{d}-1}$ is the Veronese variety and  \cref{thm:BKK} gives the special case of B\'ezout's Theorem when all the degrees are equal.

\subsection{The Veronese embedding and its geometry}\label{sec:veronese}Once we are given a family of functions (or more generally of sections of a line bundle) on a smooth manifold, a useful idea is to use these functions to embed the manifold in a large dimensional projective space. This is precisely what we have done in  \cref{sec:IGFC} with the map $\nu_E:(\C^*)^n\to \CP^{|E|-1}.$ In this section I will discuss the probabilistic implications of this idea in the world of real algebraic geometry. 

To start with, let us introduce some general setup. Let $M$ be a smooth compact manifold and $L$ be a line bundle on it. We consider a finite family $\mathcal{B}=\{\sigma_0, \ldots, \sigma_\ell\}$ of smooth sections of $L$ and the equation:
\be \label{eq:sections}\sum_{k=0}^\ell\xi_k\sigma_k(x)=0.\ee
At this point the coefficients $\xi_0, \ldots, \xi_k$ are just real numbers, but later they will become random variables. Notice that, if the line bundle $L$ is trivial, then each section is of the form $\sigma(x)=(x, s(x))$ and can be thought as a function on $M$. In the general case this only makes sense in some trivializing chart for $L$, but the zero set of a section is still a well defined object.

Using the family $\mathcal{B}=\{\sigma_0, \ldots, \sigma_\ell\}$ we can construct an associated \emph{Veronese map} $\nu_\mathcal{B}:M\to \RP^{\ell}$, defined by
$$\nu_\mathcal{B}(x)=[\sigma_0(x), \ldots, \sigma_\ell(x)].$$
The map is well defined only if the sections $\sigma_0, \ldots, \sigma_\ell$ have no common zeroes (in algebraic geometry if such a family of sections exists, the line bundle $L$ is said to be \emph{basepoint--free}). Besides this condition, here we will also assume that $\nu_\mathcal{B}:M\to \RP^\ell$ is an embedding (if such a family of sections exists, algebraic geometers say that the line bundle $L$ is \emph{very ample}).

Let now $\sigma=\sum_{k=0}^\ell \xi_k\sigma_k(x)$ and $Z(\sigma)\subset M$ be its zero set. Since we have assumed that $\nu_\mathcal{B}$ is an embedding, $Z(\sigma)$ and $\nu_\mathcal{B}(Z(\sigma))$ are homeomorphic. On the other hand, if we consider the hyperplane $\xi^{\perp}=\{\xi_0y_0+\cdots +\xi_\ell y_\ell=0\}\subset \RP^\ell$, we see that:
\be\label{eq:transform} Z(\sigma)\simeq \nu_\mathcal{B}(Z(\sigma))=\nu_\mathcal{B}(M)\cap \xi^\perp.\ee
The Veronese map transforms the geometry of zero sets of linear combinations of elements from the family $\mathcal{B}$  into the geometry of hyperplane sections of the image $\nu_\mathcal{B}(M)$. 

The special case of interest for us is given by the choice $M=\RP^n$ and $L=\mathcal{O}(d)$. 
By construction\footnote{Since the cocycle of the line bundle $\mathcal{O}(d)$, over the trivializing open cover $\{x_j\neq0\}_{j=0, \ldots, n}$ for $\RP^n$, is given by $g_{ij}([x])=\left(\frac{x_j}{x_i}\right)^d.$} every polynomial $p\in \R[x_0, \ldots, x_n]_{(d)}$ defines in a natural way a section of $\mathcal{O}(d)$ and the zero set of $p$ as a section is the same as the zero set of $p$ as a homogeneous polynomial on the projective space. In this case we consider the following family of sections:
\be\label{eq:Kostlanbasis}\mathcal{B}_{n,d}:= \left\{\sigma_\alpha=\left(\frac{d!}{\alpha_0!\cdots \alpha_n!}\right)^{\frac{1}{2}}x_0^{\alpha_0}\cdots x_n^{\alpha_n}\right\}_{|\alpha|=d}.\ee
The family $\mathcal{B}_{n,d}$, that we already met in \eqref{eq:BWbasis}, is a basis for the space of homogeneous polynomials and, since $\mathcal{O}(d)$ is very ample, the corresponding Veronese map $$\nu_{n,d}:\RP^n\to \RP^{\binom{n+d}{d}-1}$$
is an embedding. The map $\nu_{n,d}$ will also be called the \emph{Kostlan--Veronese} embedding.

As we have seen, the reason for the scaling factors in front of the monomials in \eqref{eq:Kostlanbasis} is that they make the construction invariant under the action of the orthogonal group $O(n+1)$ on the space of polynomials by change of variables (and I have chosen the name Kostlan--Veronese because E. Kostlan investigated the properties of this action). Let us look a little bit more into this construction. 

We need to make a couple of preliminary identifications. First, set $N:= {n+d\choose d}$ and define $$P_{n,d}:=\mathrm{P}(\R[x_0, \ldots, x_n]_{(d)})\simeq \RP^{N-1}.$$ We will consider the explicit isomorphism $s:\RP^{N-1}\to P_{n,d}$ given by $[e_\alpha]\mapsto [\sigma_\alpha]$ (here $e_\alpha\in \R^{N}$ is the $\alpha$--th standard basis vector). When $d=1$ we get the usual identification $[a_0,\ldots, a_n]\mapsto[a_0x_0+\cdots+a_nx_n]$ between $\RP^n$ and its dual $(\RP^n)^*=P_{n, 1}$. Notice that the pull--back under $s$ of the Bombieri--Weyl quotient metric on $P_{n,d}$ is the standard quotient metric on $\RP^{N-1}$ (this is simply because the map $s$ lifts to a linear map which sends the standard orthonormal basis to the Bombieri--Weyl one). 
We also notice that in the commutative diagram
$$\begin{tikzcd}
\RP^n \arrow[rr, "{\nu_{n,d}}"] \arrow[d, "s"'] &  & \RP^{N-1} \arrow[d, "s"] \\
{P_{n,1}} \arrow[rr, "\nu"]                     &  & {P_{n, d}}              
\end{tikzcd}$$
the map $\nu$ sends a linear form to its $d$--th power (this follows from the multinomial expansion).


Let us now consider actions of the group $O(n+1)$ on these objects. Every $g\in O(n+1)$ acts linearly on $\RP^n$, and dually on $P_{n,1}$, by isometries. The group $O(n+1)$ also acts on $P_{n,d}$ through the change of variables representation 
$$\rho:O(n+1)\to \mathrm{GL}(\R[x_0, \ldots, x_n]_{(d)}),$$ 
defined by $\rho(g)p:=p\circ g^{-1}.$  In particular, using the identification $\RP^{N-1}\simeq P_{n,d}$, the orthogonal group also acts on $\RP^{N-1}$ (we still denote by $\rho(g):\RP^{N-1}\to \RP^{N-1}$ this action). Since the action on $P_{n,d}$ with the Bombieri--Weyl metric is by isometries (\cref{sec:BWdef}), and since the pull--back of the Bombieri--Weyl metric under $s$ is the standard metric,  then the action defined in this way on $\RP^{N-1}$ with the standard metric is also by isometries.

Denoting by $V_{n,d}:=\nu_{n,d}(\RP^n)$, for every $g\in O(n+1)$ we have the commutative diagram:
\be\label{eq:commutative}\begin{tikzcd}
\RP^n \arrow[d, "{\nu_{n,d}}"'] \arrow[rr, "g"]                                &  & \RP^n \arrow[d, "{\nu_{n,d}}"]     \\
{V_{n,d}} \arrow[d, hook] \arrow[rr, "{\rho(g)|_{V_{n,d}}}"] &  & {V_{n,d}} \arrow[d, hook] \\
\RP^{N -1}\arrow[rr, "\rho(g)"]                                       &  & \RP^{{N}-1}                
\end{tikzcd}
.\ee
Endowing $V_{n,d}$ with the induced Riemannian structure, all horizontal arrows are isometries. The only nontrivial fact in the previous diagram is that $\rho(g)$ maps $V_{n,d}$ to itself: viewing the action on $P_{n,d}$, then $V_{n,d}$ is the variety of $d$--th powers of linear forms, which is preserved under change of variables (it is indeed an orbit of this action). A consequence of this is the fact that the pull--back of the quotient metric on $\RP^{N-1}$ under $\nu_{n,d}$ can be explicitly computed and, up to multiples, it equals the quotient metric on $\RP^n.$


\begin{theorem}\label{thm:veronese}The Kostlan--Veronese embedding $\nu_{n,d}:\RP^n\to \RP^{N-1}$ is a dilation of factor $d^\frac{1}{2}$. In particular, for every submanifold $Z\hookrightarrow \RP^n$ of dimension $m$ we have:
$$ \mathrm{vol}(\nu_{n,d}(Z))=d^{\frac{m}{2}}\mathrm{vol}(Z).$$
\end{theorem}
\begin{proof}Let $[x]\in \RP^n$ and $v\in T_{[x]}\RP^n$; we need to prove that $\|(D_{[x]}\nu_{n,d})v\|=d^{1/2}\|v\|.$ Without loss of generality we assume that $\|v\|=1$; we also choose the representative for $[x]$ such that $x\in S^n$ and we identify $T_{[x]}\RP^n\simeq T_xS^{n}=\{x\}^\perp.$ Pick a rotation $g\in O(n+1)$ such that $g x=(1, 0, \ldots, 0)=e_0$ and $gv=(0,1,0, \ldots, 0)=e_1$. Then, using  the fact that horizontal arrows in \eqref{eq:commutative} are isometries, we easily get
$$ \|(D_{[x]}\nu_{n,d})v\|=\|(D_{[e_0]}\nu_{n,d})e_1\|.$$
For $t\in (-\epsilon, \epsilon)$ let $\gamma(t):=[\cos t, \sin t, 0, \ldots, 0]\in \RP^n$, so that $\gamma(0)=[e_0]$ and $\dot\gamma(0)=e_1$. Then, by an explicit computation we see that
$$(D_{[e_0]}\nu_{n,d})e_1=\frac{d}{dt}\left(\nu_{n,d}(\gamma(t))\right)\bigg|_{t=0}=\frac{d}{dt}[(\cos t)^d, \sqrt{d}(\cos t)^{d-1}\sin t, \ldots]\bigg|_{t=0}=(0, \sqrt{d}, 0, \ldots, 0).$$
This proves the statement.
\end{proof}
If one omits the scaling coefficients and considers the ``standard'' Veronese embedding, the invariance under the $O(n+1)$ action gets lost, and the expression for the pull-back metric is rather complicated. For instance, in the case $n=1$ the length of the standard rational normal curve of degree $d$, i.e. the image of $[x_0, x_1]\mapsto [x_0^d, x_0^{d-1}x_1, \ldots, x_0 x_1^{d-1}, x_1^d]$, is $\log d+O(1)$ \cite{EdelmanKostlan95}.

Let us now move to the probabilistic side and assume that the coefficients in \eqref{eq:sections} are independent standard Gaussian variables. If $M$ is of dimension $m$, the expectation of the number of common zeroes of $m$ independent random equations as in \eqref{eq:sections} can be computed using the Integral Geometry Formula and the two equivalent descriptions of the uniform distribution on  $\G(n-\ell, \ell)$ that we gave in \cref{sec:randomlinear}. If $(\xi_{i,j})\in \R^{n\times (\ell+1)}$ is a matrix filled with independent standard gaussians and $g\in O(n+1)$ is a random element from the uniform distribution, then:
$$ \left\{\sum_{k=0}^\ell \xi_{1,k}y_k=\dots=\sum_{k=0}^\ell \xi_{n,k}y_k=0\right\}\sim g\RP^{n-\ell}$$
(recall that for two random variables $\xi_1, \xi_2:P\to X$, the notation ``$\xi_1\sim \xi_2$'' means that they have the same distribution). 

As a consequence, reasoning as in the proof of  \cref{thm:BKK},
\be\label{eq:randomsums}\EE\#\left\{\sum_{k=0}^\ell \xi_{1,k}\sigma_k(x)=\dots=\sum_{k=0}^\ell \xi_{n,k}\sigma_k(x)=0\right\}=\EE\#\nu_{\mathcal{B}}(M)\cap g\RP^{\ell-n}=\frac{\mathrm{vol}(\nu_{\mathcal{B}}(M))}{\mathrm{vol}(\RP^n)},
\ee
where in the last identity we have used  \cref{thm:IGF}.

In the special case $M=\RP^n$ and $L=\mathcal{O}(d)$, taking random combinations of the elements from $\mathcal{B}_{n,d}$ with coefficients which are independent standard gaussians,  we get back random Bombieri--Weyl polynomials \eqref{eq:Kostlan1}.
As in \eqref{eq:transform}, the zero set of a polynomial, written as in \eqref{eq:Kostlan1}, is homeomorphic (and isometric, up to a factor $d^{\frac{n-1}{2}}$) to $\nu_{n,d}(\RP^n)\cap \xi^{\perp}$, where now $\xi^{\perp}$ is a uniform random hyperplane in $\RP^{\binom{n+d}{d}-1}$. Since $\mathrm{vol}(\nu_{n,d}(\RP^n))=d^{\frac{n}{2}}\mathrm{vol}(\RP^n)$ (\cref{thm:veronese}), \eqref{eq:randomsums} in this case tells that the expected number of common zeroes in $\RP^n$ of $n$ independent random Bombieri--Weyl polynomials of degree $d$ equals $d^{\frac{n}{2}}$. This result can be generalized to the case when the degrees are different.
 
\begin{theorem}[Edelman--Kostlan--Shub--Smale]\label{thm:EKSS}Let $p_1, \ldots, p_n$ be random, independent homogeneous polynomials of degrees $d_1, \ldots, d_n$ and Bombieri--Weyl distributed. Then
$$ \EE\#Z(p_1,\ldots, p_n)=\left(d_1\cdots d_n\right)^{\frac{1}{2}}.$$
\end{theorem}
In the multi-degree case the proof of  \cref{thm:EKSS} uses the coarea formula -- or the Kac--Rice formula, which we will discuss later (see \cref{sec:KR}). I think that the simplest proof is the (apparently new\footnote{In the case $n=1$ the Integral Geometry approach is the one used in \cite{EdelmanKostlan95}; the case of multiple equations with different degrees is \cite[Theorem A]{Bez2}, proved using the so called ``double--fibration trick''.}) one we discuss here, which uses again the Integral Geometry Formula and  \cref{thm:veronese}.
\begin{proof} It is enough to show that for every $k=1, \ldots, n$ we have
\be\label{eq:Evol} \EE\mathrm{vol}(Z(p_1,\ldots, p_k))=\left(d_1\cdots d_k\right)^{\frac{1}{2}}\mathrm{vol}(\RP^{n-k}),\ee
and  \cref{thm:EKSS} follows with $k=n$.
We prove it by induction on $k\geq 1$. The base of the induction is given by the Integral Geometry Formula, in the space $\RP^{\binom{n+d}{d}-1}$ with the action of the group $G=O(\binom{n+d}{d})
$, and  \cref{thm:veronese}. For the inductive step one uses the independence of $p_1, \ldots, p_k$ and proceeds as follows. Using \cref{thm:veronese},
\begin{align*}\EE\mathrm{vol}(Z(p_1, \ldots, p_k))&=\underset{p_1, \ldots, p_{k-1}}{\EE}\int_{G}\mathrm{vol}\left(\nu_{d_k,n}^{-1}\left(\nu_{d_k,n}(Z(p_1,\ldots, p_{k-1}))\cap g\RP^{\binom{n+d_k}{d_k}-1}\right)\right)\mathrm{d}g\\
&=\underset{p_1, \ldots, p_{k-1}}{\EE}\int_{G}d_k^{-\frac{n-k}{2}}\mathrm{vol}\left(\nu_{d_k,n}(Z(p_1,\ldots, p_{k-1}))\cap g\RP^{\binom{n+d_k}{d_k}-1}\right)\mathrm{d}g\\
&=:(*)
\end{align*}
By the Integral Geometry Formula in $\RP^{\binom{n+d_k}{d_k}}$,
\begin{align*}
(*)&=\underset{p_1, \ldots, p_{k-1}}{\EE}d_k^{-\frac{n-k}{2}}\mathrm{vol}(\RP^{n-k})\frac{\mathrm{vol}\left(\nu_{d_k,n}(Z(p_1,\ldots, p_{k-1}))\right)}{\mathrm{vol}(\RP^{n-k+1})}\\
&=d_k^{-\frac{n-k}{2}}\frac{\mathrm{vol}(\RP^{n-k})}{\mathrm{vol}(\RP^{n-k+1})}\underset{p_1, \ldots, p_{k-1}}{\EE}\mathrm{vol}\left(\nu_{d_k,n}(Z(p_1,\ldots, p_{k-1}))\right)=:(**)\end{align*}
Finally, using again \cref{thm:veronese}, we get
\begin{align*}
(**)&=d_k^{-\frac{n-k}{2}}\frac{\mathrm{vol}(\RP^{n-k})}{\mathrm{vol}(\RP^{n-k+1})}\underset{p_1, \ldots, p_{k-1}}{\EE}d_k^{\frac{n-k+1}{2}}\mathrm{vol}\left(Z(p_1,\ldots, p_{k-1})\right)\\
&=d_k^{\frac{1}{2}}\frac{\mathrm{vol}(\RP^{n-k})}{\mathrm{vol}(\RP^{n-k+1})}\underset{p_1, \ldots, p_{k-1}}{\EE}\mathrm{vol}\left(Z(p_1,\ldots, p_{k-1})\right)=:(***),
\end{align*}
and by the inductive hypothesis,
$$(***)=\left(d_1\cdots d_k\right)^{\frac{1}{2}}\mathrm{vol}(\RP^{n-k}).$$
(Admittedly, these steps are hard to follow, but it is instructive to analyze them.)
\end{proof}
Notice that the same reasoning, together with the fact that $\mathrm{vol}(\nu_{n,d}(\CP^n))=d^n\mathrm{vol}(\CP^n)$, gives B\'ezout's Theorem (for a generic system). Keeping this in mind, \cref{thm:EKSS} tells that expected number of real zeroes of a random system of Kostlan polynomials equals the square--root of the number of generic complex solutions. 
\begin{remark}[Variances, moments and central limits]From a probabilist's point of view, \cref{thm:EKSS} is just the starting point for the understanding of the random variable \be\#_{k, n, d}:=\mathrm{vol}_{n-k}(Z(p_1, \ldots, p_k)),\ee
where $p_1, \ldots, p_k$ are independent Kostlan polynomials on $\RP^n$ (notice that, for $k\leq n$, the expectation of this random variable is given by \eqref{eq:Evol}). In the case $k=n=1$ (a single polynomial in one variable), F. Dalmao \cite{Dalmao} proved that the variance $\EE(\#_{1,1,d}-\sqrt{d})^2$ behaves asymptotically, when $d\to \infty$, as $\sigma^2 \sqrt{d}$, and used this result to deduce a Central Limit Theorem for $\#_{1,1,d}$. The main tools are the Kac--Rice formula (see \cref{propo:KR} below), Kratz–Le\'on’s version of the chaotic expansion of a random variable \cite{KratzLeon} and the Fourth Moment Theorem (this method was previously applied to the case of classical random trigonometric polynomials by J.-M. Aza\"is, F. Dalmao and J. R. Le\'on \cite{ADL}). Using different techniques, M. Ancona and T. Letendre \cite{AnconaLetendre} computed all the first $p$--th central moments $m_p(\#_{1,1,d})$ proving that, as $d\to \infty$:
$$m_p(\#_{1,1,d})= \mu_p\sigma^pd^{\frac{p}{4}} +o(d^{\frac{p}{4}}),$$
where $\mu_p$ are the moments of the standard gaussian distribution and $\sigma$ is the same constant as in Dalmao's result. Using this result, Ancona and Letendre proved a strong law of large numbers and recover Dalmao's Central Limit Theorem using the method of moments\footnote{In fact in \cite{AnconaLetendre} they study more generally the real zeros of random real sections of positive Hermitian line bundles over real algebraic curves.}. The estimate for the variance of $\#_{k,n, d}$ as $d\to \infty$ in the case $k<n$ was proved by Letendre \cite{Letendre}. The asymptotic variance and the Central Limit Theorem for all $(k,n)$ (including square systems $k=n$) was proved in a sequence of papers by Armentano,  Aza\"is, Dalmao and  L\'eon \cite{AADL, AADL2, AADL3}.

\end{remark}
\subsection{Probabilistic Intersection Theory}\label{sec:Schubert}
If we compare \cref{thm:EKSS} with the classical B\'ezout Theorem, we notice that the number $\sqrt{d_1\cdots d_n}$ cannot be obtained by means of cohomological operations -- as it happens instead for the number $d_1\cdots d_n$, which can be obtained by computing in the cohomology ring of the projective space. In this section I want to explain how to still interpret the probabilistic computation in an appropriate ring, as proposed by P. B\"urgisser and myself  in \cite{PSC}, and further explored by the two of us together with P. Breiding and L. Mathis in \cite{BBLM}, shifting the point of view from ``generic'' to ``random'' in the context of intersection theory. 

The elements of this ring are convex bodies of a special type, called \emph{zonoids}: they are Hausdorff limits of Minkowski sums of segments \cite{bible}. Given a vector space $V$, we denote by $\mathcal{Z}(V)$ the set of (formal differences of) zonoids in $V$, which we center at the origin, and by
$$\mathcal{A}(V):=\bigoplus_{k=0}^{\dim(V)}\mathcal{Z}(\Lambda^k V).$$
In \cite{BBLM} we proved that it is possible to endow $\mathcal{A}(V)$ with a ring structure that turns it into a graded, commutative algebra, which we called the \emph{zonoid algebra}. We called the product operation the ``wedge'' and denoted it by ``$\wedge$''. If $K_1$ and $K_2$ are zonoids, one should think of $K_1\wedge K_2$ as convex--body analogue of the wedge of differential forms (see \cite[Theorem 4.1]{BBLM} for more insights). The ring operation enjoys some interesting properties, allowing to reinterpret classical operations on convex bodies. For instance, if $K_1, \ldots, K_m\in \mathcal{Z}(\R^m)$, then $K_1\wedge\cdots\wedge K_m$ is a segment of length $\mathrm{MV}(K_1, \ldots, K_m)$ (the mixed volume).

Let now $\ell_0\in G(k,n)$ be a fixed point and identify $T_{\ell_0}G(k,n)\simeq \R^k\otimes \R^{n-k}$. Now we associate to every nice\footnote{We do not need $\Omega$ to be a smooth submanifold. For instance a submanifold--complex with finite volume is enough.} submanifold  $\Omega\subseteq G(k,n)$ a zonoid
$$K_\Omega\subset \Lambda^{\mathrm{codim}(\Omega)}(\R^k\otimes \R^{n-k}).$$
If $\Omega=\{\ell\}$ is a point (a submanifold whose codimension is equal to the dimension of $G(k,n)$) we define $K_{\{\ell\}}$ to be the centered segment in $\Lambda^{k(n-k)}(\R^k\otimes \R^{n-k})\simeq \R$ of length $\mathrm{vol}(M)^{-1}.$ Let me also explain how to build $K_{\Omega}$ when $\Omega$ is a hypersurface. First, for every $x\in \Omega$ we denote by $s(x)\subset (T_x\Omega)^\perp\simeq \R$ the unit centered segment. Using the group action, we may assume that all the segments $\{s(x)\}_{x\in \Omega}$ lie in $T_{\ell_0}G(k,n).$ Now we define
$$K_\Omega:=\int_{O(n)\times O(n-k)}\int_{\Omega}h\cdot s(x)\mathrm{d}x\mathrm{d} h,$$
where ``$\mathrm{d}x$''  and ``$\mathrm{d}h$'' denote integration with respect to the uniform measure on $\Omega$ and $O(k)\times O(n-k)$, respectively.
In other words, we ``average'' these segments over the group action of the stabilizer and over the manifold $\Omega$ (the reasons why this construction makes sense is because the integral is a limit of sums). The construction when $\mathrm{codim}(\Omega)>1$ is similar and involves choosing a unit normal parallelogram (see Leo Mathis' PhD thesis \cite{Leothesis} for more details on the construction).

This map, from submanifolds to zonoids, has the property that, if $\Omega_1, \ldots, \Omega_\nu$ are submanifolds with $\mathrm{codim}(\Omega_i)=c_i$ and such that $c_1+\cdots +c_\nu=m$, then, by \cite[Theorem 4.1.4]{Leothesis},
\be\label{eq:schubert3}K_{\Omega_1}\wedge \cdots \wedge K_{\Omega_\nu}=\left(\int_{O(n)^\nu}\# \left(g_1\Omega_1\cap \cdots \cap g_\nu\Omega_1\right)\mathrm{d}g_1\cdots \mathrm{d}g_\nu\right) \cdot K_{\{\ell_0\}}.\ee
Computing in $\mathcal{A}(V)$ one can get average intersection numbers (notice the analogy with classical intersection theory, with which this construction shares the main features).

\begin{example}In a similar way as classical B\'ezout Theorem can be formulated as an identity in $H^{*}(\CP^n)$, \cref{thm:EKSS} can be formulated instead as an identity in the zonoid algebra. If $\Omega\subset \RP^n$ is a \emph{random} Bombieri--Weyl hypersurface of degree $d$, the associated zonoid is 
$$K_{\Omega}=\frac{\sqrt{d}}{2\pi}B^n,$$ 
where $B\subset \R^n$ denotes the unit ball. Therefore, if $\Omega_1, \ldots, \Omega_n\subset \RP^n$ are random, independent and Bombieri--Weyl distributed of  degree $d_1, \ldots, d_n$, we have 
 $$\EE\left( \#\Omega_1\cap\cdots\cap \Omega_n\right)\cdot K_{[e_0]}=K_{\Omega_1}\wedge\cdots \wedge K_{\Omega_n}=\frac{\sqrt{d_1}}{2\pi}B\wedge\cdots \wedge \frac{\sqrt{d_n}}{2\pi}B=\sqrt{d_1\cdots d_n} \cdot K_{[e_0]},$$
since $B\wedge\cdots \wedge B$ is a segment of length $\mathrm{MV}(B, \ldots, B)=n!\mathrm{vol}(B)$ (here $[e_0]=[1, 0, \ldots,0]$ is a point in $\RP^n$). This is the ring--theoretical interpretation of \cref{thm:EKSS}.
\end{example}

\subsubsection{Probabilistic Schubert Calculus}In general, integrals like \eqref{eq:schubert3} can be used to define what in  \cite{PSC} we called ``Probabilistic Schubert Calculus''. For example, consider the probabilistic Schubert problem of counting the expectation of the number of lines  hitting $2n-2$ random copies of $\RP^{n-2}$ inside $\RP^n$. When $n=3$ this is the problem fo computing the expectation of the number of real lines hitting four random lines in $\RP^3$ (generically the number of complex lines is two, but over the Reals it can be either zero or two).

Since the uniform distribution on the Grassmannian is induced by the orthogonal group, this corresponds to intersecting $2n-2$ random copies of the Schubert variety $\Omega\subset \mathbb{G}(1,n)$ of lines hitting a fixed $\RP^{n-2}$, where the randomness comes by acting on each of these copies with a different copy of the orthogonal group. In \cite{PSC} we computed the asymptotic of the expected number of real solutions to this problem, as $n\to \infty$:
\be \label{eq:PSCas}\EE\#\left(\Omega\cap \cdots\cap \Omega\right)=(2n-2)!\mathrm{vol}\left(K_{\Omega}\right)=\frac{8}{3\pi^{\frac{5}{2}}\sqrt{n}}\left(\frac{\pi^2}{4}\right)^n\cdot \left(1+O(n^{-1})\right).\ee
Here the convex body $K_\Omega=\frac{1}{4\pi}B^2\otimes B^{n-2}$ is the tensor product of balls in the space of matrices $ \R^2\otimes  \R^{n-2}$, a convex body whose support function depends only on the singular value of the matrix, and \eqref{eq:PSCas} follows from the asymptotic computation of its volume. We do not know how to exactly compute these integrals in general, but using techniques from convex geometry we are often able to estimate their asymptotic behavior, see \cite{PSC2}.

\begin{remark}Using different techniques (a vector bundle version of \cref{propo:KR}), one can also formulate random Chern-classes computations. For instance, together with S. Basu, E. Lundberg and C. Peterson \cite{BLLP}, we computed the asymptotic of the expected number $E_n$ of real lines on a random, Bombieri--Weyl hypersurface of degree $(2n-3)$ in $\RP^n$ (when $n=3$ this is the classical question of counting lines on a cubic surface).
Denoting by $C_n$ the  number of complex lines on a generic hypersurface of degree $2n-3$ in $\CP^n$ , we showed that 
$$\lim_{n\to \infty}\frac{\log E_n}{\log C_n}=\frac{1}{2}.$$
(By the way, the expected number of real lines on a random cubic surface is $6\sqrt{2}-3$, see \cite{BLLP, AEMBM}). The problem of counting lines on hypersurfaces is quite interesting, as one can also define a generic real count, involving signs \cite{finkh, OkTel}.  This number is the Euler number of some vector bundle, which in this specific case equals $(2n-3)!!$.
\end{remark}
\section{Topology of random hypersurfaces}\label{sec:topology}
In this section I would like to discuss a probabilistic version of the first part of Hilbert's Sixteenth problem, formulated by D. Hilbert at the ICM in Paris in 1900 \cite{H16trans}:

\emph{``The upper bound of closed and separate branches of an algebraic curve of degree $d$ was decided by Harnack [...]. It seems to me that a thorough investigation of the relative positions of the upper bound for separate branches is of great interest, and similarly the corresponding investigation of the number, shape and position of the sheets of an algebraic surface in space.''}

 We can interpret the original formulation of the problem in modern language as the study of the Betti numbers of a nonsingular hypersurface of degree $d$ in $\RP^n$ and the way this hypersurface embeds in projective space. I refer the reader to \cite{Wilson} for an overview on this problem. 

Denoting by $P_{n,d}$ the space of homogenous polynomials of degree $d$ and by $\Sigma_{n,d}\subset P_{n,d}$ the set of polynomials whose real zero set is singular, the set $P_{n,d}\setminus \Sigma_{n,d}$ consists of several open connected components, which we called chambers:
$$P_{n,d}\setminus \Sigma_{n,d}=\bigsqcup_{\tau}C_\tau.$$
By \cref{thm:disc}, for every chamber $C_\tau$ there exists a smooth hypersurface $Z_\tau$ such that for every $p\in C_\tau$ we have $(\RP^n, Z(p))\simeq (\RP^n, Z_\tau).$ In particular, the function $p\mapsto b(Z(p))$ is constant on each chamber. Again by \cref{thm:disc}, this function is bounded and the maximum value that it attains is estimated by Thom--Milnor bound\cite{Milnor}:
$$ b(Z(p))\leq (2d)^n.$$
For what concerns the ``arrangement'' of the components of the zero set of $p$, less is known (even the precise formulation of the problem is more delicate). For example, the zero set of a curve of degree $d$ in the plane cannot contain a sequence of too many nested ovals, one inside the other\footnote{At most it can have $\lfloor\frac{d}{2}\rfloor$ such ovals; curves of this type are called \emph{maximally nested}.}. What is clear after \cite{OreKha} is that the number of possible pairs $(\RP^n, Z(p))$ up to diffeomorphisms grows exponentially fast as $d\to \infty$ and this fact rules out the possibility of a case--by--case study.

A new perspective comes into the picture if we endow the space of polynomials with a probability distribution. Then, we can ask for the same questions but for  ``random'' hypersurfaces. For instance, we can wonder about the expectation of the function $p\mapsto Z(p)$ or, even more, we can try to understand if in the space of polynomials there are some special chambers which have more probability than others. 

For the rest of this section ``random'' will always mean with respect to the Bombieri--Weyl distribution. In this case, the first problem (computing $\EE b(Z(p))$) was studied by D. Gayet and J.--Y. Welschinger in a sequence of papers \cite{GaWe0, GaWe1, GaWe2}. The second problem (detecting whether there are special chambers that have more probability than others) by D. N. Diatta and myself in \cite{DiattaLerario}, where we showed that ``most'' hypersurfaces of degree $d$ are isotopic to hypersurfaces of degree $O(\sqrt{d\log d})$. I will try now to give an account of these results.
\subsection{The Betti numbers of a random algebraic hypersurface}\label{sec:averagebetti}The very first result involving topological invariants of a random hypersurface $Z(p)\subset \RP^n$ was proved by S. S. Podkorytov \cite{Po}, who used an integral representation of the Euler characteristic to prove a closed formula for the expectation of $\chi(Z(p))$, yielding for $n$ odd\footnote{If $n$ is even, $Z(p)\subseteq \RP^n$ is with probability one an odd--dimensional manifold, therefore with zero Euler characteristic.}:
$$\EE \chi(Z(p))=\frac{\Gamma\left(\frac{n}{2}\right)}{\sqrt{\pi}\Gamma\left(\frac{n+1}{2}\right)}d^{\frac{n}{2}}+O\left(d^{\frac{n}{2}-1}\right).$$
Using different techniques (related to the ideas discussed in \cref{sec:veronese}) P. B\"urguisser \cite{Bu} generalized this result to the case of random complete intersections, giving also an explicit formula for the expectation of the curvature polynomial of $Z(p)$ in $\RP^n$. 

The work of F. Nazarov and M. Sodin \cite{NazarovSodin1}, on the nodal components of a random spherical harmonic, introduced a new technique in the subject (what today is called the ``barrier method'', which I will discuss in this context in \cref{sec:sqrt}). In fact I discovered the whole field thanks to P. Sarnak's handwritten letter \cite{Sarnak}, which points out the relevance of \cite{NazarovSodin1} for the emerging field of ``random algebraic geometry''. I really recommend the reader to read this letter. 

The goal of this section is to explain the ideas behind the proof of the following result, proved in  \cite{GaWe1, GaWe2}. The setting from  \cite{GaWe1, GaWe2} is more general, and the authors deal with random hypersurfaces of a smooth real projective manifold, but the case of $\RP^n$ already contains the main features.
\begin{theorem}[Gayet--Welschinger]\label{thm:GaWe}There exist $c_{1,n}, c_{2,n}>0$ such that if $p\in\R[x_0, \ldots, x_n]_{(d)}$ is a random Bombieri--Weyl polynomial, then
\be\label{eq:bettiaverage}c_{1, n}d^\frac{n}{2}\leq \EE b(Z(p))\leq c_{2,n} d^\frac{n}{2}.\ee
\end{theorem}

\subsubsection{The upper bound and the Kac--Rice formula}\label{sec:KR}
The main idea for the proof of the upper bound in \eqref{eq:bettiaverage} is to use a \emph{random} version of Morse theory. To be more precise: fix a Morse function $g:\RP^n\to \R$ and for every polynomial $p\in P_{n,d}\setminus \Sigma_{n,d}$ consider the restriction 
$$g|_{Z(p)}:Z(p)\to \R.$$ Clearly this is a smooth function on $Z(p)$. It is not difficult to show that, except for those $p$ belonging to a zero--measure set, and for large $d>0$, the function $g|_{Z(p)}$ is also Morse. In particular, by Morse inequalities \cite[Theorem 5.2]{MilnorMorse},
$$\EE\#b(Z(p))\leq \EE\#\mathrm{crit}(g|_{Z(p)}).$$
Observe that the set $\mathrm{crit}(g|_{Z(p)})$ is described by a system of random equations:
\be\label{eq:reduce}\textrm{crit}(g|_{Z(p)})=\{x\in \RP^n\,|\, p(x)=0\quad \textrm{and}\quad T_xZ(p)\subset \ker d_xg\}.\ee
The problem reduces therefore to estimate the expectation of the cardinality of the set of solutions of a system of random equations. How do we proceed from here?

The tool to use in these cases is the so--called \emph{Kac--Rice formula}, introduced first by M. Kac \cite{kac43} and S. O. Rice \cite{Rice} for the study of the number of zeros of random functions. In order to formulate the statement, let $U\subseteq \R^n$ and $G:U\to \R^n$ be a ``random map''. In the cases we will be interested in, each component of $G$ will be a gaussian combination of some fixed smooth functions, as in \eqref{eq:sections}, but let me be loose on this point and explain instead the main ingredients. For every $y\in \R^n$ we have a random variable $(G(y), JG(y))\in \R^n\times \R^{n\times n}$ obtained by simply evaluating at $y\in U$ the map $G$ and its Jacobian. We assume that this random variable has a smooth density, which we denote by $h(\cdot, \cdot, y):\R^n\times \R^{n\times n}\to \R$
$$ \PP\{(G(y), JG(y))\in A\}=\int_{A} h(v, w, y) \mathrm{d}v\mathrm{d}w.$$
We also assume that $0$ is a regular value of $G$ with probability one. The precise hypotheses in the statement of next theorem can  be found in \cite[Theorem 11.2.1]{AdlerTaylor}. 
\begin{theorem}[Kac--Rice formula]\label{propo:KR} Let $G:U\to \R^n$ be a random map as above. Then for every Borel set $V\subseteq U$
\be \EE\#\left(G^{-1}(0)\cap V\right)=\int_{V}\int_{\R^{n\times n}}|\det(w)|h(0, w,y)\mathrm{d}w \mathrm{d}y.\ee
\end{theorem}
The function $\rho(y):=\int_{\R^{n\times n}}|\det(w)|h(0, w,y)\mathrm{d}w$ is called the \emph{Kac--Rice density} of the zeroes of $G$, since the expectation of the numbers of zeroes of $G$ in $V\subseteq U$ is written as an integral of $\rho$ on $V$.
It is difficult to underestimate the impact of this formula on the development of the theory: essentially, every time that we are able to reduce the problem to counting ``points'' we are in the position of using it. 

This is precisely what is happening in \eqref{eq:reduce}. First, after passing to local coordinates, one can reduce the problem to \cref{propo:KR} and show that there exists a Kac--Rice density $\rho_{n,d}:\R^n\to \R$ for the critical points of $g|_{Z(p)}$:
$$ \label{eq:krgw}\EE \#\mathrm{crit}(g|_{Z(p)})=\int_{\RP^n} \rho_{n, d}(x)\mathrm{vol}_{\RP^n}(\mathrm{d} x).$$
Using the $O(n+1)$--invariance of the Bombieri--Weyl probability distribution, one shows that $\rho_{n,d}$ is constant; it is therefore sufficient to evaluate its value at the point $[e_0]=[1, 0, \ldots, 0]$. Working with the coordinates $y_i:=x_i/x_0$ on the open set $\{x_0\neq 0\}\simeq \R^n$, and using again the $O(n+1)$--invariance, we can assume that $d_{[e_0]}g=\lambda d_{[e_0]}y_1$ for some $\lambda\neq 0$. 

Denote now by $\tilde{p}:\R^n\to \R$ the function $\tilde{p}(y_1, \ldots, y_n):=p(1, y_1, \ldots, y_n)$ and notice that near zero we have an expansion:
\be\label{eq:expansion}\tilde{p}( y_1, \ldots, y_n)=\xi_{0}+d^{\frac{1}{2}}\sum_{j=1}^n \xi_j y_j+\left(\frac{d(d-1)}{2}\right)^{\frac{1}{2}}\left(\sum_{j=1}^n\xi_{jj}y_j^2+\sqrt{2}\sum_{i\neq j}\xi_{ij}y_iy_j\right)+O(\|y\|^3)\ee
(we have relabeled the gaussian variables). 
In particular one can compute exactly the Kac--Rice density at zero, using the recipe we gave above:
\be \label{eq:krexp}\rho_{n, d}\equiv d^{\frac{n}{2}}\cdot \int_{\mathrm{Sym}(n-1, \R)}|\det(Q)|\,\mathrm{\gamma_{\mathrm{GOE}}(\mathrm{d}Q})+O\left(d^{\frac{n}{2}-1}\right),\ee
where $\gamma_{\mathrm{GOE}}$ is the GOE measure on the space of symmetric matrices (see \cref{example:GOE}). The appearance of the random matrix model is due to the fact that the quadratic part of \eqref{eq:expansion}, up to a factor, is a Bombieri--Weyl quadratic form, whose corresponding symmetric matrix is a GOE matrix. Similar expansions can be proved if one only looks at critical points of a given index $k$, integrating over the set of symmetric matrices with signature $(k, n-1-k)$. The upper bound in \eqref{eq:bettiaverage} now follows from the finiteness of the integral $\int_{\mathrm{Sym}(n-1, \R)}|\det(Q)|\,\mathrm{\gamma_{\mathrm{GOE}}(\mathrm{d}Q})$. 
%

\subsubsection{Lower bounds and rescaling limits}\label{sec:sqrt} Before moving to the proof of the lower bound in \eqref{eq:bettiaverage} we will discuss a  powerful idea, which goes back to the work F. Nazarov and M. Sodin \cite{NazarovSodin1, NazarovSodin2}, and suggests to look at the limiting behaviour of a random polynomial on a small disk, by rescaling the variable in an appropriate way, as we have done in \eqref{eq:expansion}.

Given a homogeneous polynomial $p\in \R[x_0, \ldots, x_n]_{(d)}$, it will be convenient for the rest of the discussion to look at it as a function\footnote{If $d=2\ell$ is even we can eve regard $p$ as a function on $\RP^n$. In fact the function
$$[x_0, \ldots, x_n]\mapsto \frac{p(x)}{(x_0^2+\cdots +x_n^2)^\ell}$$
is well defined (since $x_0^2+\cdots +x_n^2$ has no zeroes on $\RP^n$) and has the same zero set as $p$.} on $S^n$ (by restriction), rather then a section of $\mathcal{O}(d)$ on $\RP^n$: the zero set of $p$ on $S^n$ is a double cover of the zero set of $p$ in $\RP^n$ and one can easily recover the structure of one from the other. Moreover, if we are looking at the local behaviour of $p$, it makes no difference to regard it as a section of $\mathcal{O}(d)$ or as a function on $S^n$, as we can always compose it with the local inverse of the covering map.

Let now $e_0=(1, 0, \ldots, 0)\in S^n$ (or any other point, it doesn't matter for the discussion, by the orthogonal invariance of the Bombieri--Weyl measure) and consider the unit disk $D\subset \R^n$ and the family of embeddings $\eta_d:D\to S^n$ given by
$$\eta_d(y):=\frac{e_0+\frac{2}{\sqrt{d}}y}{\|e_0+\frac{2}{\sqrt{d}}y\|}=e_0+\frac{y}{\sqrt{d}}+O\left(\frac{\|y\|^2}{d}\right).$$
 We denote by $D_d\subset S^n$ the image of $\eta_d$: notice that this is a Riemannian disk in the sphere with center $e_0$ and radius $\sin(2d^{-1/2}).$ Given a polynomial $p\in \R[x_0, \ldots, x_n]_{(d)}$, consider the map $f_d:D\to \R$ defined by the diagram
 $$\begin{tikzcd}
D \arrow[r, "\eta_d"] \arrow[rr, "f_d"', bend right] & D_d \arrow[r, "p"] & \R
\end{tikzcd}.$$ 
Since $\eta_d$ is a diffeomorphism onto its image, the pairs $(D, Z(f_d))$ and $(D_d, Z(p))$ are diffeomorphic. If $p$ is a random polynomial, then $f_d$ is a random variable with values in $C^{\infty}(D, \R)$. Does this random variable exhibit some special properties in the large $d$ limit? The next result expresses a crucial property of the Bombieri--Weyl distribution.
\begin{proposition}\label{propo:limitBW}Let $U\subset C^{\infty}(D, \R)$ be a nonempty open set for the $C^k$--topology such that $\PP\{f_d\in \partial U\}=0$. Then there exists $c(U)>0$ such that
$$\lim_{d\to \infty}\PP\{f_d\in U\}=c(U).$$
\end{proposition}
I will sketch the proof of the weaker statement $ \liminf_{d\to \infty}\PP\{f_d\in U\}>0$, which does not require the assumption $\PP\{f_d\in \partial U\}=0$. This will suffice for the purpose of proving the lower bound in \cref{thm:GaWe}; the existence of the limit is more delicate and it is proved in \cite[Theorem 23]{LerarioStecconi}.
\begin{proof}
The crucial observation is the expansion (I am omitting the big--Oh term for simplicity):
\begin{align}f_d(y)&=\sum_{k=0}^d\sum_{|\beta|=k}\left(\frac{d!}{(d-k)!\beta_1!\cdots \beta_n!}\right)^{\frac{1}{2}}\xi_{k\beta}\left(\frac{y_1}{\sqrt{d}}\right)^{\beta_1}\cdots\left(\frac{y_n}{\sqrt{d}}\right)^{\beta_n}\\
\label{eq:sum}&=\sum_{k=0}^d\sum_{|\beta|=k}\left(\frac{d!}{(d-k)!d^k}\right)^{\frac{1}{2}}\xi_{k\beta}\left(\frac{1}{\beta_1!\cdots \beta_n!}\right)^{\frac{1}{2}}y_1^{\beta_1}\cdots y_n^{\beta_n}.
\end{align}
Notice that, $\left(\frac{d!}{(d-k)!d^k}\right)^{\frac{1}{2}}\to 1$ as $d\to \infty$. Therefore from this expansion we see that, once $d_1\in \mathbb{N}$ is fixed, the probability distribution induced on the coefficients of monomials of degree at most $d_1$ converges to a fixed nondegenerate gaussian distribution. The proof just aims at proving that given the open set $U$ we can find $d_1=d_1(U)>0$ such that only the coefficients of monomials of degree at most $d_1$ matters for the probability of $U$ and the tail coefficients (those of degree at least $d_1+1$) act just as a little perturbations.

Let us make things a little more rigorous. Pick $g\in U.$ Since $U$ is open in the $C^k$--topology, there exists $\delta>0$ such that  $B_{C^k(D, \R)}(g, \delta)\subseteq U$. Since polynomials are dense in the $C^k$--topology on the disk (by Weierstrass' approximation Theorem), there exists a polynomial $q$ such that $\|q-g\|_{C^k(D,\R)}<\frac{\delta}{2}.$ Let $d_0=\deg(q)$ and for $d\geq d_1\geq d_0$ split the sum \eqref{eq:sum} as 
$$f_d=\sum_{k=0}^{d_1}(\cdots)+\sum_{k=d_1+1}^d(\cdots)=f_{\leq d_1}+f_{>d_1}.$$
Observe that for every $\alpha=(\alpha_1, \ldots, \alpha_n)$ with $|\alpha|\leq r$
\be\label{eq:seriesd}\EE \max_{y\in D}\left|\partial_{y_1}^{\alpha_1}\cdots\partial_{y_n}^{\alpha_n}f_{>d_1}(y)\right|\leq \sum_{k=d_1+1}^\infty\sum_{|\beta|=k}\sqrt{\frac{2}{\pi}}\left(\frac{d!}{(d-k)!d^k}\right)^{\frac{1}{2}}\left(\frac{1}{\beta_1!\cdots \beta_n!}\right)^{\frac{1}{2}}\alpha_1!\cdots \alpha_n!\ee
Since the series in \eqref{eq:seriesd} is converging, by Markov's inequality, there exists $c_1>0$ and $d_1=d_1(U)>0$ such that
\be\label{eq:E1}\PP\left\{\|f_{>d_1}\|_{C^r(D, \R)}\leq \frac{\delta}{4}\right\}\geq 1- \frac{c_1}{\delta}>0.\ee
Now $d_1$ is fixed and, since all norms on finite dimensional spaces are equivalent, there exists $c_2>0$ such that 
\be \label{eq:E2}\PP\left\{\|f_{\leq d_1}-q\|_{C^r(D, \R)}<\frac{\delta}{4}\right\}\geq c_2>0.\ee
Since the two events in \eqref{eq:E1} and \eqref{eq:E2} are independent (they involve disjoint sets of independend gaussian variables), there is a positive probability that they happen at once. This means that there exists $c_3:=c_1c_2>0$ such that as $d\to \infty$
$$\PP\left\{\|f_d-q\|_{C^r(D, \R)}\leq \|f_{\leq d_1}-q\|_{C^r(D, \R)}+\|f_{>d_1}\|_{C^r(D, \R)}<\frac{\delta}{2}\right\}\geq c_3.$$
This implies that the smooth map $f_d$ belongs to $B(g, \delta)\subseteq U$ with probability at least $c_3>0$ and concludes the proof.
\end{proof}
With \cref{propo:limitBW} available the proof of the lower bound in \eqref{eq:bettiaverage} uses the so called ``barrier method'' introduced in \cite{NazarovSodin1}, and  goes as follows. Pick your favorite compact hypersurface $Z\subset D$ and consider the set 
$$U_{Z}=\{f:D\to \R\,|\, (D, Z(f))\simeq (D, Z)\}\subset C^{\infty}(D, \R).$$
The set $U_{Z}$ is an open set for the $C^1$--topology and by \cref{propo:limitBW}, $\lim \PP\{f_d\in U_{Z}\}=c(U_{Z})>0.$
Put now in the sphere $S^n$ at least $c_3 d^{\frac{n}{2}}>0$ disjoint disks $B_1, \ldots, B_{c_3d^{n/2}}$ of radius $\sin(2d^{-1/2})\approx 2d^{-1/2}$ (we can chose so many disjoint disks by a doubling argument). Then
\begin{align*}\EE b(Z(p))&\geq\EE b(\{\textrm{components of $Z(p)$ entirely contained in $\cup B_j$}\})\\
&=\sum_{j=1}^{c_3 d^{\frac{n}{2}}}\EE b(\{\textrm{components of $Z(p)$ entirely contained in $B_j$}\})\\
&\geq \sum_{j=1}^{c_3 d^{\frac{n}{2}}} b(Z)c(U_{Z})\geq c_{1,n}d^{\frac{n}{2}}.
\end{align*}
This proves the lower bound in \cref{thm:GaWe}.

\subsection{The global structure: low degree approximation}I would like to discuss now the so--called ``low degree approximation theorem'', a results from \cite{DiattaLerario}, showing that most polynomials $p$ of degree $d$ can be approximated by polynomials $q$ of degree $O(\sqrt{d \log d})$ such that $(\RP^n, Z(p))\simeq (\RP^n, Z(q)).$ In other words, the ``arrangement'' of most hypersurfaces of degree $d$ is the same of some hypersurface of degree $O(\sqrt{d\log d}).$ 

In order to state the result, for every $\ell\in \mathbb{N}$ such that $d-\ell\in 2\mathbb{N}$, consider the space $\|x\|^{d-\ell}\cdot P_{n, \ell}$, consisting of polynomials $p\in P_{n,d}$ of the form $p(x)=\|x\|^{d-\ell}q(x)$, for some $q\in P_{n, \ell}$. This is a subspace of $P_{n,d}$ isomorphic to $P_{n, \ell}$. Since $\|x\|^{d-\ell}$ has no real projective zeroes, for what concerns topology, polynomials from this space are polynomials of degree $d$ whose real zero set is given by a polynomial of degree $\ell$. Denote by 
$$\tau_\ell:P_{n,d}\to \|x\|^{d-\ell}\cdot P_{n, \ell}$$
the orthogonal projection\footnote{This is the projection with respect to any scalar product on the space of polynomials which is invariant by orthogonal change of variables. Of course the Bombieri--Weyl scalar product has this property, but there are other scalar products with this invariance, see \cref{sec:projection}.} onto this subspace.

\begin{theorem}[Diatta--Lerario]\label{thm:LDA}Let $p$ be a Bombieri--Weyl polynomial of degree $d$. There exists $c_n>0$ such that for all $d\in \N$ and $\ell\in \mathbb{N}$ such that $d-\ell\in 2\mathbb{N}$, 
\be\label{eq:LDA}\PP\bigg\{(\RP^n, Z(p))\simeq (\RP^n, Z(\tau_{\ell}(p)))\bigg\}\geq 1-c_nd^{c_n}e^{-\frac{\ell^2}{c_n d}}.\ee
\end{theorem}
From this statement we see that, if we choose $\ell=b\sqrt{d\log d}$ with $b>0$ large enough, the exponential part in \eqref{eq:LDA} dominates the polynomial part and the above probability goes to one as $d\to \infty$ (with a polynomial rate). In other words, most of the Bombieri--Weyl measure is concentrated near chambers containing polynomials  whose truncation to degree $O(\sqrt{d\log d})$ have zero sets diffeomorphic to the original polynomial. Notice that the majority of the chambers does not have this property, but the chambers that fail to have this property all together have Bombieri--Weyl measure that goes to zero as $d\to \infty$. 

If $\ell=d^b$ with $b\in (\frac{1}{2}, 1)$, the rate of convergence is of the order $e^{-d^{O(1)}}$  and if $\ell=bd$, with $b\in (0,1)$, it is of the order $e^{-O(d)}$. These estimates can be used to give a bound on the probability of special arrangements. The first result in this direction is due to D. Gayet and J.--Y. Welschinger \cite{GayetWelschingerrarefaction}, where the authors prove that the Bombieri--Weyl probability of the set of maximal curves of degree $d$  (i.e. curves with $O(d^2)$ many ovals) decays exponentially fast as $d\to \infty.$ Using \cite{DiattaLerario} one can prove exponential rarefaction of maximal projective hypersurfaces in dimension $n$ and also for more general arrangements (e.g. maximally nested ones).  I let the reader speculate on these statements.

\subsubsection{The decomposition into spherical harmonics}\label{sec:projection}Before moving to the proof \cref{thm:LDA}, let me explain the approximation procedure (the projection $\tau_\ell$ from the previous section) from a different point of view, showing that it can be seen as multidimensional analogue of the truncation of a Fourier expansion. As above, we work with polynomials as functions on the sphere $S^n$; we use the notation $\mathcal{P}_{n,d}:=\R[x_0, \ldots, x_n]_{(d)}|_{S^n}$ to stress this fact. To start with, recall that the space $L^2(S^n)$ can be written as
\be\label{eq:sphericaldeco}L^2(S^n)=\overline{\bigoplus_{k\in \mathbb{N}}V_{n, k}}^{L^2(S^n)},\ee
where, for every $k\in \N$ the space $V_{n, k}$ denotes the space of \emph{spherical harmonics of degree $k$} (i.e. eigenfunctions of the spherical Laplacian with eigenvalues $\lambda_k=-k(n-k+1)$). What is important for us is that $V_{n,k}$ coincides with the set of functions on the sphere which are restrictions of \emph{harmonic} homogenous polynomials of degree $k$ (i.e. homogeneous polynomials of degree $k$ which are in the kernel of the Laplacian on $\R^{n+1}$). When $n=1$, this decomposition is just the standard Fourier decomposition, since $V_{1, k}=\mathrm{span}\{\cos(k\theta), \sin(k\theta)\}.$

If in \eqref{eq:sphericaldeco} we take the sum only over $k\in \{k\leq d,\, \textrm{$d-k$ even}\}$ we get precisely the space of homogeneous polynomials (restricted to the sphere):
\be \label{eq:irr}\mathcal{P}_{n,d}=\bigoplus_{k\leq d,\, d-k \in 2\N}V_{n,k}.\ee
This decomposition is orthogonal with respect to the $L^2(S^n)$ scalar product. Remarkably, it is also orthogonal with respect to the Bombieri--Weyl scalar product, and in fact with respect to \emph{any} $O(n+1)$--invariant scalar product. This a consequence of the fact that \eqref{eq:irr} is precisely the decomposition into irreducible subspaces under the $O(n+1)$--action (this is a classical fact, see \cite[Theorem 4.37]{Lerarionotes} for a detailed proof). 

As a consequence of this fact, given $\ell\leq d$ with $d-\ell\in 2\mathbb{N}$, we see that the projection $\tau_{\ell}$ defined above is nothing but the orthogonal projection
$$\tau_\ell:\mathcal{P}_{n,d}=\bigoplus_{k\leq d,\, d-k \in 2\N}V_{n,k}\to \bigoplus_{k\leq \ell,\, \ell-k \in 2\N}V_{n,k}=\mathcal{P}_{n, ,\ell}.$$

\begin{remark}It is now a good point to explain why the Bombieri--Weyl distribution is not the unique invariant one and how to classify invariant gaussian measures, as done by E. Kostlan in \cite{Kostlan95}. Recall from \cref{remark:gaussianscalar} that there is a one--to--one correspondence between nondegenerate, centered gaussian distributions and scalar products; invariant scalar products give rise to invariant distributions.  Since $V_{n,k}$ is irreducible for the orthogonal change of variables representation, By Schur's Lemma, on each $V_{n,k}$ there is only one invariant scalar product: up to multiples this is the $L^2(S^n)$ one. Moreover, since for $k_1\neq k_2$ the representations $V_{n,k_1}$ and $V_{n, k_2}$ are not isomorphic, every invariant scalar product on $\mathcal{P}_{n,d}$ is of the form:
\be\label{eq:invariantsp} \langle \cdot, \cdot\rangle=\sum_{d-k\in 2\mathbb{N}} \beta(k) \langle  \cdot,\cdot \rangle_{L^2(S^n)}|_{V_{n, k}},\ee
for some collection of ``weights'' $\{\beta(k)>0\}_{d-k\in 2\N}$. It follows that there is a whole family of invariant scalar products on $\R[x_0, \ldots, x_n]_{(d)}$ (and therefore of invariant measures), depending on $\lfloor \frac{d}{2}\rfloor +1$ parameters (the weights).
\end{remark}

\begin{remark}
The reader might wonder if a statement similar to \cref{thm:GaWe} holds for other invariant measures. For purely spherical harmonics, F. Nazarov and M. Sodin \cite{NazarovSodin1} proved that the expectation of the number of connected components of the zero set of a random spherical harmonic of degree $d$ on $S^2$  is $\Theta(d^2)$. This paper contains the ``barrier method'' idea: all the above mentioned results on lower bounds of probabilities are essentially just technical improvements of this idea. Using this idea, together with E. Lundberg \cite{Lerarioshsp} we proved that for the Gaussian measure induced on the space of polynomials of degree $d$ from the $L^2(S^n)$ scalar product, $\EE b_0(Z(p))=\Theta(d^n)$. More generally, together with Y. Fyodorov and E. Lundberg, in \cite{LerarioFLL} we defined the notion of \emph{coherent} distribution, which is essentially a family of invariant measures on the space of polynomials of degree $d$ such that, as $d\to \infty$, the weights \eqref{eq:invariantsp} have some limiting behaviour. More precisely,  a coherent family of probability distributions on the space of polynomials is given by the the choice of weights $\{\beta_{d}(k)\}_{d-k\in 2\N}$ in \eqref{eq:invariantsp} satisfying, as $d\to \infty$,
$$\exists \lambda\in (0,1),\, \exists\psi\in L^2(0, \infty)\quad\textrm{such that}\quad \beta_{d}(xd^{\lambda})d^{\lambda}\to \psi(x).$$
For instance, the Bombieri--Weyl distributions are coherent with parameter $\lambda=\frac{1}{2}$ and $\psi$ is a gaussian function; the $L^2(S^n)$ distributions are coherent with $\lambda=1$ and $\psi$ is the characteristic function of the unit interval. For every $\lambda\in (0,1)$ there are coherent families with parameter $\lambda$.
 For a coherent invariant family with parameter $\lambda\in (0,1)$, we proved that $\EE b_0(Z(p))=\Theta(d^{\lambda n})$ (this result ``interpolates'' between \cite{GaWe1, GaWe2} and \cite{Lerarioshsp}).
\end{remark}
\subsubsection{Proof of \cref{thm:LDA}}
Let us now explain the ideas behind the proof of \cref{thm:LDA}. We start with some differential topology. Given $p\in C^\infty(S^n, \R)$ such that $p=0$ is a regular equation on the sphere, if we take a little perturbation $\tilde{p}$ in the $C^1$--topology, the pair  $(S^n,Z(p))$ and $(S^n, Z(\tilde{p}))$ are diffeomorphic. The fact that, if $p$ is a polynomial, it is possible to estimate how large can this perturbation be, follows from a theorem of C. Raffalli \cite{raffalli}. 
\begin{theorem}[Raffalli]\label{thm:raffalli}For every $p\in P_{n,d}$
\be\label{eq:BWdist}\mathrm{dist}_{\mathrm{BW}}(p, \Sigma_{n,d})=\min_{x\in S^n}\left(|p(x)|^2+\frac{\|\mathrm{proj}_{T_{x}S^n}(\nabla p(x))\|^2}{d}\right)^{\frac{1}{2}}.\ee
\end{theorem}
This result generalizes the classical Eckart--Young Theorem, which gives a closed formula for the distance, in the Frobenius norm, between a square matrix  and the ``discriminant'' set of matrices with determinant zero. The proof uses the orthogonal invariance of the Bombieri--Weyl structure and the reader can try to figure it out by herself. 

Using the explicit formula for the distance from the discriminant given by \cref{thm:raffalli}, one sees that, if $\|p-f\|_{C^1(S^n,\R)}\leq \mathrm{dist}_{\mathrm{BW}}(p, \Sigma_{n,d})$, then the equation $(1-t)p+tf=0$ on the sphere is regular for every $t\in  [0,1]$ and, by Thom's Isotopy Lemma (a variation on \cref{thm:disc}) we have $ (S^n, Z(p))\simeq (S^n, Z(f)).$ In order to prove the theorem it would therefore be ``enough'' to prove that the Bombieri--Weyl measure of the set
\be\label{eq:enough}U_{ d, \ell}:=\bigg\{\|p-\tau_{\ell}(p)\|_{C^1(S^n,\R)}\leq \mathrm{dist}_{\mathrm{BW}}(p, \Sigma_{n,d})\bigg\}\ee
is ``large''  (how large depends on the choice of $\ell$, as in the statement), since the zero set of every element $p\in U_{d, \ell}$ and of its truncation $\tau_{\ell}(p)$ are diffeomorphic. 

Unfortunately, it is not easy to work directly with the event in the parentheses of \eqref{eq:enough} and the strategy is to produce instead a sequence of intermediate inequalities
$$\|p-\tau_{\ell}(p)\|_{C^1(S^n,\R)}\underset{(1)}{\leq} (*)\underset{(2)}{\leq}(**)\underset{(3)}{\leq} \mathrm{dist}_{\mathrm{BW}}(p, \Sigma_{n,d}),$$
each of which has some geometric interpretation and holds with some probability. 
 
 The first such inequality (1) replaces the $C^1$--norm with the  $q$--\emph{Sobolev norm}, which is induced by the scalar prodcut
\be\label{eq:sobolev}\langle\cdot, \cdot\rangle_{H^q}:=\sum_{d-k\in 2\mathbb{N}} k^{2q} \langle  \cdot,\cdot \rangle_{L^2(S^n)}|_{V_{n, k}}.\ee
Notice that this is a special case of an invariant scalar product \eqref{eq:invariantsp}. Using the fact that the space of spherical harmonics is a reproducing kernel Hilbert space, one can produce $L^2$ estimates on the derivatives of its elements. Since on each space of spherical harmonics the $H^q$ norm is a multiple of the $L^2$ (as in \eqref{eq:sobolev}), we can use these estimates to show that, with the choice $q=\frac{n+1}{2}$ and for some $a_1>0$,
$$\|p-\tau_{\ell}(p)\|_{C^1(S^n,\R)}\underset{(1)}{\leq} a_1d^{\frac{1}{2}}\|p-\tau_{\ell}(p)\|_{H^q}.$$

The second inequality (2) is a variation of the gaussian concentration inequality, which estimates the probability of events of the form
\be\label{eq:E}\bigg\{\|\mathrm{proj}_E(p)\|\leq t\|p\|\bigg\}\ee
where  $P$ is a gaussian space, $\|\cdot\|$ denotes the norm for the scalar product inducing the gaussian structure and $E$ is some low--codimension subspace. For instance, if $E$ is of codimension one, the gaussian measure concentrates near a neighborhood of $E$ of size $\dim(P)^{-\frac{1}{2}}$. One can even let the dimension of $E$ be a fraction (close to 1) of the dimension of $P$ and still get concentration inequalities, see \cite{Artstein}. Observe now that $\|p-\tau_{\ell}(p)\|_{H^q}=\|\mathrm{proj}_{E_\ell}(p)\|_{H^q}$, where $E_\ell$ is the orthogonal complement (in the Bombieri--Weyl norm) of $\mathcal{P}_{n, \ell}$. In order to proceed we would like instead to estimate the probability of the event
\be\label{eq:El}\bigg\{\|\mathrm{proj}_{E_\ell}(p)\|_{H^q}\leq t\|p\|_{\mathrm{BW}}\bigg\}.\ee
Compare \eqref{eq:E} with \eqref{eq:El}. They look very similar except for two facts: (a) for the regime we are interested in ($\ell$ as close as possible to $\sqrt{d}$) the space $E_\ell$ has \emph{high} codimension in $\mathcal{P}_{n,d}$ and (b) on the left hand side of \eqref{eq:El} we have the $H^q$ norm and not the BW one (which is the one inducing the gaussian distribution). However, the fact that the $H^q$ norm introduces some ``weights'', makes it indeed possible to treat our problem as a concentration problem and to estimate our probability, for some $a_2>0$, as follows:
$$\PP\bigg\{\|\mathrm{proj}_{E_\ell}(p)\|_{H^q}\underset{(2)}{\leq} t\|p\|_{\mathrm{BW}}\bigg\}\geq1-a_2\frac{d^{-\frac{3n}{2}+2}\ell^{2q+n-2}e^{-\frac{\ell^2}{d}}}{t^2}.$$

The last inequality (3) has to deal with estimating the volume (i.e. the probability) of a small neighborhood of $\Sigma_{n,d}$ in $\mathcal{P}_{n,d}$. This fits into a general interesting problem: given a hypersurfaces $\Sigma$ in a gaussian space $P$, with $\Sigma$ contained in an algebraic set $Z$,  can one estimate the probability of the ``neighborhood'' 
$$\bigg\{p\in P\,\bigg|\,\mathrm{dist}(p, \Sigma)\leq t\|p\|\bigg\}$$ 
as a function of $t>0$, the degree of $Z$ and the dimension of $P$? The answer is given by the following result \cite{BCL} (the proof uses again ideas related to integral geometry). 
\begin{proposition}[B\"urgisser--Cucker--Lotz]Let $P$ be a gaussian space of dimension $N$, $Z\subset P$ be a hypersurface of degree $D$ and $\Sigma\subseteq Z$. For all $t\leq \frac{1}{(2D+1)N} $
\be\label{eq:distanceD}\PP\bigg\{p\in P\,\bigg|\,\mathrm{dist}(p, \Sigma)\leq t\|p\|\bigg\}\leq 4e DNt .\ee
\end{proposition}
Using the fact that $\Sigma_{n,d}$ is contained in an algebraic set, i.e. the intersection of the space of real polynomials (which has dimension $N=O(d^n)$) with the complex discriminant (which has degree $D=O(d^n)$), in our case \eqref{eq:distanceD} gives $a_3, a_4>0$ such that for all $t\leq a_4d^{-2n}$,
$$\PP\bigg\{t\|p\|_{\mathrm{BW}}\underset{(3)}{\leq}\mathrm{dist}_{\mathrm{BW}}(p, \Sigma_{n,d})\bigg\}\geq1-a_3td^{2n}.$$
Collecting now back the inequalities (1), (2), (3) and tuning the parameter $t$, gives 
$$\PP\bigg\{\|p-\tau_{\ell}(p)\|_{C^1(S^n,\R)}\leq \mathrm{dist}_{\mathrm{BW}}(p, \Sigma_{n,d})\bigg\}\geq1-c_nd^{c_n}e^{-\frac{\ell^2}{c_nd}}.$$
This proves \cref{thm:LDA}.

 \begin{remark}M. Ancona \cite{Anconararefaction} has generalized the previous proof from \cite{DiattaLerario} and proved exponential rarefaction of maximal hypersurfaces in real algebraic varieties. The hypersurfaces are given by zero sets of sections of an ample real Hermitian holomorphic line bundle, and the measure is the same considered by \cite{GaWe1}.  \end{remark}

\bibliographystyle{alpha}
\bibliography{literature}

\newcommand{\noop}[1]{}
\begin{thebibliography}{AEMBM21}

\bibitem[AADL18]{AADL}
D.~Armentano, J-M. Aza\"{\i}s, F.~Dalmao, and J.~R. Le\'{o}n.
\newblock Asymptotic variance of the number of real roots of random polynomial
  systems.
\newblock {\em Proc. Amer. Math. Soc.}, 146(12):5437--5449, 2018.

\bibitem[AADL21]{AADL2}
D.~Armentano, J.-M. Aza\"{\i}s, F.~Dalmao, and J.~R. Le\'{o}n.
\newblock Central limit theorem for the number of real roots of {K}ostlan
  {S}hub {S}male random polynomial systems.
\newblock {\em Amer. J. Math.}, 143(4):1011--1042, 2021.

\bibitem[AADL22]{AADL3}
Diego Armentano, Jean-Marc Aza\"{\i}s, Federico Dalmao, and Jos\'{e}~R.
  Le\'{o}n.
\newblock Central limit theorem for the volume of the zero set of
  {K}ostlan-{S}hub-{S}male random polynomial systems.
\newblock {\em J. Complexity}, 72:Paper No. 101668, 22, 2022.

\bibitem[ADL16]{ADL}
Jean-Marc Aza\"{\i}s, Federico Dalmao, and Jos\'{e}~R. Le\'{o}n.
\newblock C{LT} for the zeros of classical random trigonometric polynomials.
\newblock {\em Ann. Inst. Henri Poincar\'{e} Probab. Stat.}, 52(2):804--820,
  2016.

\bibitem[AEMBM21]{AEMBM}
Rida Ait El~Manssour, Mara Belotti, and Chiara Meroni.
\newblock Real lines on random cubic surfaces.
\newblock {\em Arnold Math. J.}, 7(4):541--559, 2021.

\bibitem[AL21]{AnconaLetendre}
Michele Ancona and Thomas Letendre.
\newblock Roots of {K}ostlan polynomials: moments, strong law of large numbers
  and central limit theorem.
\newblock {\em Ann. H. Lebesgue}, 4:1659--1703, 2021.

\bibitem[Anc24]{Anconararefaction}
Michele Ancona.
\newblock Exponential rarefaction of maximal real algebraic hypersurfaces.
\newblock {\em J. Eur. Math. Soc. (JEMS)}, 26(4):1423--1444, 2024.

\bibitem[Art02]{Artstein}
Shiri Artstein.
\newblock Proportional concentration phenomena on the sphere.
\newblock {\em Israel J. Math.}, 132:337--358, 2002.

\bibitem[AT07]{AdlerTaylor}
R.~J. Adler and J.~E. Taylor.
\newblock {\em Random fields and geometry}.
\newblock Springer Monographs in Mathematics. Springer, New York, 2007.

\bibitem[Ati82]{Atyiah1}
M.~F. Atiyah.
\newblock Convexity and commuting {H}amiltonians.
\newblock {\em Bull. London Math. Soc.}, 14(1):1--15, 1982.

\bibitem[B\"07]{Bu}
Peter B\"urgisser.
\newblock Average {E}uler characteristic of random real algebraic varieties.
\newblock {\em C. R. Math. Acad. Sci. Paris}, 345(9):507--512, 2007.

\bibitem[BBLM22]{BBLM}
Paul Breiding, Peter B\"{u}rgisser, Antonio Lerario, and L\'{e}o Mathis.
\newblock The zonoid algebra, generalized mixed volumes, and random
  determinants.
\newblock {\em Adv. Math.}, 402:Paper No. 108361, 57, 2022.

\bibitem[BCL08]{BCL}
Peter B\"{u}rgisser, Felipe Cucker, and Martin Lotz.
\newblock The probability that a slightly perturbed numerical analysis problem
  is difficult.
\newblock {\em Math. Comp.}, 77(263):1559--1583, 2008.

\bibitem[BCR98]{BCR:98}
Jacek Bochnak, Michel Coste, and Marie-Fran{\c{c}}oise Roy.
\newblock {\em Real algebraic geometry}, volume~36 of {\em Ergebnisse der
  Mathematik und ihrer Grenzgebiete (3)}.
\newblock Springer-Verlag, Berlin, 1998.

\bibitem[BL20]{PSC}
Peter B\"{u}rgisser and Antonio Lerario.
\newblock Probabilistic {S}chubert calculus.
\newblock {\em J. Reine Angew. Math.}, 760:1--58, 2020.

\bibitem[BLLP19]{BLLP}
Saugata Basu, Antonio Lerario, Erik Lundberg, and Chris Peterson.
\newblock Random fields and the enumerative geometry of lines on real and
  complex hypersurfaces.
\newblock {\em Math. Ann.}, 374(3-4):1773--1810, 2019.

\bibitem[Dal15]{Dalmao}
Federico Dalmao.
\newblock Asymptotic variance and {CLT} for the number of zeros of {K}ostlan
  {S}hub {S}male random polynomials.
\newblock {\em C. R. Math. Acad. Sci. Paris}, 353(12):1141--1145, 2015.

\bibitem[DL22]{DiattaLerario}
Daouda~Niang Diatta and Antonio Lerario.
\newblock Low-degree approximation of random polynomials.
\newblock {\em Found. Comput. Math.}, 22(1):77--97, 2022.

\bibitem[Ehr51]{Ehresmann}
Charles Ehresmann.
\newblock Les connexions infinit\'{e}simales dans un espace fibr\'{e}
  diff\'{e}rentiable.
\newblock In {\em Colloque de topologie (espaces fibr\'{e}s), {B}ruxelles,
  1950,}, pages 29--55. ,, 1951.

\bibitem[EK95]{EdelmanKostlan95}
Alan Edelman and Eric Kostlan.
\newblock How many zeros of a random polynomial are real?
\newblock {\em Bull. Amer. Math. Soc. (N.S.)}, 32(1):1--37, 1995.

\bibitem[FK13]{finkh}
S.~Finashin and V.~Kharlamov.
\newblock Abundance of real lines on real projective hypersurfaces.
\newblock {\em Int. Math. Res. Not. IMRN}, (16):3639--3646, 2013.

\bibitem[FLL15]{LerarioFLL}
Yan~V. Fyodorov, Antonio Lerario, and Erik Lundberg.
\newblock On the number of connected components of random algebraic
  hypersurfaces.
\newblock {\em J. Geom. Phys.}, 95:1--20, 2015.

\bibitem[GW11]{GayetWelschingerrarefaction}
Damien Gayet and Jean-Yves Welschinger.
\newblock Exponential rarefaction of real curves with many components.
\newblock {\em Publ. Math. Inst. Hautes \'{E}tudes Sci.}, (113):69--96, 2011.

\bibitem[GW14a]{GaWe1}
D.~Gayet and J.-Y. Welschinger.
\newblock Lower estimates for the expected {B}etti numbers of random real
  hypersurfaces.
\newblock {\em J. Lond. Math. Soc.}, 90:105--120, 2014.

\bibitem[GW14b]{GaWe0}
Damien Gayet and Jean-Yves Welschinger.
\newblock What is the total {B}etti number of a random real hypersurface?
\newblock {\em J. Reine Angew. Math.}, 689:137--168, 2014.

\bibitem[GW16]{GaWe2}
Damien Gayet and Jean-Yves Welschinger.
\newblock Betti numbers of random real hypersurfaces and determinants of random
  symmetric matrices.
\newblock {\em J. Eur. Math. Soc. (JEMS)}, 18(4):733--772, 2016.

\bibitem[Hat02]{Hatcher}
Allen Hatcher.
\newblock {\em Algebraic topology}.
\newblock Cambridge University Press, Cambridge, 2002.

\bibitem[Hir94]{Hirsch}
Morris~W. Hirsch.
\newblock {\em Differential topology}, volume~33 of {\em Graduate Texts in
  Mathematics}.
\newblock Springer-Verlag, New York, 1994.
\newblock Corrected reprint of the 1976 original.

\bibitem[How93]{Howard}
R.~Howard.
\newblock The kinematic formula in {R}iemannian homogeneous spaces.
\newblock {\em Mem. Amer. Math. Soc.}, 106(509):vi+69, 1993.

\bibitem[Kac43]{kac43}
M.~Kac.
\newblock On the average number of real roots of a random algebraic equation.
\newblock {\em Bull. Amer. Math. Soc.}, 49:314--320, 1943.

\bibitem[KL97]{KratzLeon}
Marie~F. Kratz and Jos\'{e}~R. Le\'{o}n.
\newblock Hermite polynomial expansion for non-smooth functionals of stationary
  {G}aussian processes: crossings and extremes.
\newblock {\em Stochastic Process. Appl.}, 66(2):237--252, 1997.

\bibitem[Kos93]{Kostlan95}
E.~Kostlan.
\newblock On the distribution of roots of random polynomials.
\newblock In {\em From {T}opology to {C}omputation: {P}roceedings of the
  {S}malefest ({B}erkeley, {CA}, 1990)}, pages 419--431. Springer, New York,
  1993.

\bibitem[Ler22]{Lerarionotes}
Antonio Lerario.
\newblock Lecture notes on metric algebraic geometry, 2022.

\bibitem[Let19]{Letendre}
Thomas Letendre.
\newblock Variance of the volume of random real algebraic submanifolds.
\newblock {\em Trans. Amer. Math. Soc.}, 371(6):4129--4192, 2019.

\bibitem[LL15]{Lerarioshsp}
Antonio Lerario and Erik Lundberg.
\newblock Statistics on {H}ilbert's 16th problem.
\newblock {\em Int. Math. Res. Not. IMRN}, (12):4293--4321, 2015.

\bibitem[LM20]{PSC2}
Antonio Lerario and L{\'e}o Mathis.
\newblock Probabilistic schubert calculus: Asymptotics.
\newblock {\em Arnold Mathematical Journal}, 2020.

\bibitem[LS21]{LerarioStecconi}
Antonio Lerario and Michele Stecconi.
\newblock Maximal and typical topology of real polynomial singularities.
\newblock {\em Ann. Inst. Fourier}, 2021.

\bibitem[Mat]{Leothesis}
Leo Mathis.
\newblock {T}he handbook of zonoid calculus --- hdl.handle.net.
\newblock \url{https://hdl.handle.net/20.500.11767/129410}.
\newblock [Accessed 09-Jul-2023].

\bibitem[Mil63]{MilnorMorse}
J.~Milnor.
\newblock {\em Morse theory}.
\newblock Annals of Mathematics Studies, No. 51. Princeton University Press,
  Princeton, N.J., 1963.
\newblock Based on lecture notes by M. Spivak and R. Wells.

\bibitem[Mil64]{Milnor}
J.~Milnor.
\newblock On the betti numbers of real varieties.
\newblock {\em Proc. Amer. Math. Soc.}, 15:275--280, 1964.

\bibitem[NS09]{NazarovSodin1}
Fedor Nazarov and Mikhail Sodin.
\newblock On the number of nodal domains of random spherical harmonics.
\newblock {\em Amer. J. Math.}, 131(5):1337--1357, 2009.

\bibitem[NS16]{NazarovSodin2}
F.~Nazarov and M.~Sodin.
\newblock Asymptotic laws for the spatial distribution and the number of
  connected components of zero sets of {G}aussian random functions.
\newblock {\em Zh. Mat. Fiz. Anal. Geom.}, 12(3):205--278, 2016.

\bibitem[OK00]{OreKha}
S.~Yu. Orevkov and V.~M. Kharlamov.
\newblock Growth order of the number of classes of real plane algebraic curves
  as the degree grows.
\newblock {\em Zap. Nauchn. Sem. S.-Peterburg. Otdel. Mat. Inst. Steklov.
  (POMI)}, 266(Teor. Predst. Din. Sist. Komb. i Algoritm. Metody. 5):218--233,
  339, 2000.

\bibitem[OT14]{OkTel}
Ch. Okonek and A.~Teleman.
\newblock Intrinsic signs and lower bounds in real algebraic geometry.
\newblock {\em J. Reine Angew. Math.}, 688:219--241, 2014.

\bibitem[Pod98]{Po}
S.~S. Podkorytov.
\newblock On the {E}uler characteristic of a random algebraic hypersurface.
\newblock {\em Zap. Nauchn. Sem. S.-Peterburg. Otdel. Mat. Inst. Steklov.
  (POMI)}, 252(Geom. i Topol. 3):224--230, 252--253, 1998.

\bibitem[Raf14]{raffalli}
Christophe Raffalli.
\newblock Distance to the discriminant.
\newblock {\em preprint on arXiv}, 2014.
\newblock https://arxiv.org/abs/1404.7253.

\bibitem[Ric44]{Rice}
S.~O. Rice.
\newblock Mathematical analysis of random noise.
\newblock {\em Bell System Tech. J.}, 23:282--332, 1944.

\bibitem[Sar11]{Sarnak}
P.~Sarnak.
\newblock Letter to {B}. {G}ross and {J}. {H}arris on ovals of random planes
  curve.
\newblock {\em available at
  \url{http://publications.ias.edu/sarnak/section/515}}, 2011.

\bibitem[Sch14]{bible}
Rolf Schneider.
\newblock {\em Convex bodies: the {B}runn-{M}inkowski theory}, volume 151 of
  {\em Encyclopedia of Mathematics and its Applications}.
\newblock Cambridge University Press, Cambridge, expanded edition, 2014.

\bibitem[SS93]{Bez2}
M.~Shub and S.~Smale.
\newblock Complexity of {B}\'ezout's theorem {II}: volumes and probabilities.
\newblock In F.~Eyssette and A.~Galligo, editors, {\em Computational Algebraic
  Geometry}, volume 109 of {\em Progress in Mathematics}, pages 267--285.
  Birkh\"auser, 1993.

\bibitem[tbMWN00]{H16trans}
David~Hilbert (translated~by Mary Winton~Newson).
\newblock Mathematical problems, lecture delivered before the international
  congress of mathematicians at paris in 1900, 1900.

\bibitem[Wil78]{Wilson}
George Wilson.
\newblock Hilbert's sixteenth problem.
\newblock {\em Topology}, 17(1):53--73, 1978.

\end{thebibliography}

\end{document}